\documentclass[a4paper,reqno]{amsart}
\usepackage{amsmath,amsfonts,amssymb,amsthm}
\usepackage{graphicx} 
\usepackage{xy}
\xyoption{all}
\usepackage[colorlinks, linkcolor=blue!72,anchorcolor=orange,
    citecolor=red,urlcolor=Emerald, bookmarksopen,bookmarksdepth=2]{hyperref}
\usepackage[dvipsnames]{xcolor}
\usepackage{tikz}
\usepackage{cleveref}
\usepackage{url}
\usetikzlibrary{matrix,positioning,decorations.markings,arrows,decorations.pathmorphing,backgrounds,fit,positioning,shapes.symbols,chains,shadings,fadings,calc}
\tikzset{->-/.style={decoration={  markings,  mark=at position #1 with
    {\arrow{>}}},postaction={decorate}}}
\tikzset{-<-/.style={decoration={  markings,  mark=at position #1 with
    {\arrow{<}}},postaction={decorate}}}

\theoremstyle{plain}
\newtheorem{theorem}{Theorem}[section]
\newtheorem{lemma}[theorem]{Lemma}
\newtheorem{proposition}[theorem]{Proposition}
\newtheorem{corollary}[theorem]{Corollary}
\theoremstyle{definition}
\newtheorem{definition}[theorem]{Definition}

\newtheorem{remark}[theorem]{Remark}
\newtheorem{example}[theorem]{Example}
\newtheorem{notation}[theorem]{Notation}
\newtheorem{notation-remark}[theorem]{Notation-Remark}

\numberwithin{equation}{section}

\newcommand{\Hom}{\operatorname{Hom}}
\newcommand{\End}{\operatorname{End}}

\newcommand{\thick}{\operatorname{thick}}

\newcommand{\add}{\operatorname{add}}
\newcommand{\proj}{\operatorname{proj}}
\newcommand{\inj}{\operatorname{inj}}
\renewcommand{\ker}{\operatorname{ker}}
\newcommand{\coker}{\operatorname{coker}}

\renewcommand{\mod}{\operatorname{mod}}
\newcommand{\Fac}{\operatorname{Fac}}
\newcommand{\Sub}{\operatorname{Sub}}

\newcommand{\emod}[2]{#1\text{-}\mod #2}
\newcommand{\uemod}[2]{\underline{#1\text{-}\mod} #2}
\newcommand{\oemod}[2]{\overline{#1\text{-}\mod} #2}
\newcommand{\AR}[1]{\tau_{[#1]}}

\def\k{\mathbf{k}}
\def\C{\mathcal{C}}
\def\D{\mathcal{D}}
\def\T{\mathcal{T}}
\def\F{\mathcal{F}}
\def\H{\mathcal{H}}

\def\X{\mathcal{X}}
\def\Y{\mathcal{Y}}

\def\II{\mathbf{I}}

\def\MM{\mathbf{M}}

\def\PP{\mathbf{P}}
\def\QQ{\mathbf{Q}}
\def\RR{\mathbf{R}}
\def\VV{\mathbf{V}}
\def\WW{\mathbf{W}}
\def\XX{\mathbf{X}}
\def\YY{\mathbf{Y}}
\def\ZZ{\mathbf{Z}}
\def\TT{\mathbf{T}}
\def\FF{\mathbf{F}}

\def\pp{\mathbf{p}}
\def\ii{\mathbf{i}}
\def\ppm{\mathbf{p}_m}
\def\iim{\mathbf{i}_m}

\def\tp{\AR{m}\operatorname{-tiltp}A}
\def\sil{(m+1)\operatorname{-silt}A}
\def\ft{\operatorname{f-}s\operatorname{-tors}A}

\title{Tilting theory for extended module categories}

\dedicatory{Dedicated to Professor Bin Zhu on the occasion of his 60th birthday}

\author{Yu Zhou}
\address{School of Mathematical Sciences,
Beijing Normal University,
100875 Beijing,
China}
\email{yuzhoumath@gmail.com}

\keywords{extended module category, $s$-torsion pair, Auslander-Reiten theory, $\tau$-tilting, silting complex}

\thanks{The work was supported by National Natural Science Foundation of China (Grants No. 12271279, 12031007).}

\begin{document}

\begin{abstract}
    In extended hearts of bounded $t$-structures on a triangulated category, we provide a Happel-Reiten-Smal{\o} tilting theorem and a characterization for $s$-torsion pairs. Applying these to $m$-extended module categories, we characterize torsion pairs induced by $(m+1)$-term silting complexes. After establishing Auslander-Reiten theory in extended module categories, we introduce $\AR{m}$-tilting pairs and show bijections between $\AR{m}$-tilting pairs, $(m+1)$-term silting complexes, and functorially finite $s$-torsion pairs.
\end{abstract}

\maketitle

\section*{Introduction}

Tilting theory has occupied a central position in the representation theory of algebras since the early seventies \cite{BGP,APR,BB,HR,B}. One of the significant results is that each (classical) tilting module gives rise to a functorially finite torsion pair \cite{AS,S}. Support $\tau$-tilting modules were introduced by Adachi, Iyama and Reiten \cite{AIR}, completing the theory in two important aspects. First, in terms of mutation: in contrast to classical tilting modules, where an almost complete tilting module may have one or two complements, an almost complete support $\tau$-tilting module has exactly two complements. Second, the support $\tau$-tilting modules correspond one-to-one with all functorially finite torsion pairs. 

In the work \cite{AIR}, the support $\tau$-tilting modules are also shown to correspond one-to-one with the 2-term silting complexes. Silting complexes were introduced by Keller and Vossieck \cite{KV} to classify bounded t-structures in the bounded derived category of a Dynkin quiver. K\"{o}nig and Yang \cite{KY} extended this result to the case of general finite-dimensional (non-positive differential graded) algebras, proving that silting complexes correspond one-to-one with bounded t-structures with length heart. In contrast to 2-term silting complexes, which correspond to torsion pairs in the module category, Gupta \cite{G} showed that general silting complexes correspond to torsion pairs in certain truncated subcategories of the derived category. We refer to these truncated subcategories as extended module categories.

In this paper, we first consider a general framework. Fix a positive integer $m$. Let $\D$ be a triangulated category, and let $(\D^{\leq0},\D^{\geq0})$ be a bounded $t$-structure on $\D$. We call 
$$\D^{[-(m-1),0]}=\D^{\leq 0}\cap\D^{\geq -(m-1)}$$
the $m$-extended heart of $(\D^{\leq0},\D^{\geq0})$. This subcategory is closed under extensions and is, therefore, an extriangulated category in the sense of \cite{NP}. A pair $(\T,\F)$ of full subcategories of $\D^{[-(m-1),0]}$ is called a torsion pair if $\Hom(\T,\F)=0$ and $\D^{[-(m-1),0]}=\T*\F$. A torsion pair $(\T,\F)$ is called an $s$-torsion pair if $\Hom(\T,\F[-1])=0$. Applying a main result of \cite{AET} to $\D^{[-(m-1),0]}$, one obtains a bijection between the set of bounded $t$-structures $(\C^{\leq 0},\C^{\geq 0})$ on $\D$, satisfying $\D^{\leq -m}\subseteq\C^{\leq 0}\subseteq\D^{\leq 0}$, and the set of $s$-torsion pairs in $\D^{[-(m-1),0]}$, see \Cref{prop:bi} for more details. Moreover, we have the following generalization of Happel-Reiten-Smal{\o} tilting.

\begin{theorem}[\Cref{prop:hrs}]
    For any $s$-torsion pair $(\T,\F)$ in $\D^{[-(m-1),0]}$, $(\F[m],\T)$ is an $s$-torsion pair in the $m$-extended heart of the bounded $t$-structure corresponding to $(\T,\F)$. 
\end{theorem}

We also provide an alternative criterion for a torsion pair to be an $s$-torsion pair. Specifically, a torsion pair $(\T,\F)$ is an $s$-torsion pair if and only if $\T$ is closed under $m$-factors  if and only if $\F$ is closed under $m$-subobjects, see \Cref{def:sfac}, \Cref{closed} and \Cref{prop:fac}.

We then apply these results to bounded derived categories of finite-dimensional algebras. Let $A$ be a finite-dimensional algebra over a field $\k$ and $\mod A$ be the category of finitely generated right $A$-modules. The $m$-extended module category $\emod{m}{A}$ is the full subcategory of the bounded derived category $D^b(\mod A)$ of $\mod A$, given as 
$$\emod{m}{A}=\{\XX\in D^b(\mod A)\mid H^i(\XX)=0,\ \forall i\notin [-(m-1),0]\}.$$
It is shown in \cite{G} that $(m+1)$-term silting complexes $\PP\in K^b(\proj A)$ one-to-one correspond to functorially finite $s$-torsion pairs $(\T(\PP),\F(\PP))$ in $\emod{m}{A}$. We show that both $\T(\PP)$ and $\F(\PP)$ have enough projective objects and enough injective objects, see \Cref{prop:tpp1}. As a consequence, we have the following characterization of this torsion pair, where $\nu$ is the Nakayama functor and $H^{[-(m-1),0]}$ is a truncation functor.

\begin{theorem}[\Cref{prop:genTP}]
    Let $\PP$ be an $(m+1)$-term silting complex. Then
    $$\T(\PP)=\Fac_{m}\left(H^{[-(m-1),0]}(\PP)\right)\text{ and }\F(\PP)=\Sub_{m}\left(H^{[-(m-1),0]}(\nu\PP[-1])\right).$$
\end{theorem}

We also show that the number of non-isomorphic indecomposable projective objects in $\T(\PP)$ (resp. $\F(\PP)$) is the same as the number of non-isomorphic indecomposable injective objects in $\T(\PP)$ (resp. $\F(\PP)$), see \Cref{cor:num}.

The notion of Auslander-Reiten triangles on an extension-closed subcategory of a triangulated category, or more generally, Auslander-Reiten extriangles, or equivalently, almost split extensions on an extriangulated category, were introduced in \cite{J,ZZ1,INP}, see also \cite{LN}. We show that the extended module category $\emod{m}{A}$ has Auslander-Reiten triangles (although the bounded derived category $D^b(\mod A)$ may not). 

\begin{theorem}[\Cref{thm:AR}]
    Let $\ZZ$ be an indecomposable object in $\emod{m}{A}$. If $\ZZ$ is not projective in $\emod{m}{A}$, there is an Auslander-Reiten triangle in $\emod{m}{A}$
    \[\AR{m}(\ZZ)\to\YY\to\ZZ\to\AR{m}(\ZZ)[1].\]
    If $\ZZ$ is not injective in $\emod{m}{A}$, there is an Auslander-Reiten triangle in $\emod{m}{A}$
    \[\ZZ\to\WW\to\AR{m}^{-}(\ZZ)\to\ZZ[1].\]
\end{theorem}
Note that these Auslander-Reiten translations $\AR{m}$ and $\AR{m}^-$ (see \Cref{def:AR} for their constructions) are different from the higher Auslander-Reiten translations $\tau_m$ and $\tau_m^-$ introduced in \cite{I,I2}, see \Cref{rmk:Iyama} for a preliminary discussion on their relationship.

Finally, we generalize $\tau$-tilting theory from $\mod A$ to $\emod{m}{A}$. An object $\XX$ in $\emod{m}{A}$ is called positive $\AR{m}$-rigid if
$$\Hom(\XX,\AR{m}(\XX)[j])=0,\ j\leq 0.$$
A pair $(\XX,P)$ of $\XX\in\emod{m}{A}$ and $P\in\proj A$ is called $\AR{m}$-tilting if 
\begin{enumerate}
    \item $\XX$ is positive $\AR{m}$-rigid,
    \item $\Hom(P,\XX[j])=0$ for any $j\leq 0$, and
    \item ${}^{\bot_{\leq 0}}(\AR{m}(\XX))\cap P^{\bot_{\leq 0}}\subseteq\Fac_{m}(\XX)$.
\end{enumerate}
We obtain the following one-to-one correspondences between $\AR{m}$-tilting pairs, functorially finite $s$-torsion pairs and $(m+1)$-term silting complexes.

\begin{theorem}[\Cref{thm:bi}]
    There are bijections between
    \begin{enumerate}
        \item the set of isoclasses of basic $\AR{m}$-tilting pairs in $\emod{m}{A}$,
        \item the set of functorially finite $s$-torsion pairs in $\emod{m}{A}$,
        \item the set of isoclasses of basic $(m+1)$-term silting complexes in $K^b(\proj A)$.
    \end{enumerate}
\end{theorem}

As mentioned previously, a bijection from (3) to (2) has already been established in \cite{G}, sending $\PP$ to $(\T(\PP),\F(\PP))$.

Recently, there have also been some works that apply higher Auslander-Reiten theory to generalize $\tau$-tilting theory and establish connections with higher torsion theory and silting theory, see \cite{AHJKPT, JJ,MM,M,RV,ZZ2} and so on.

The paper is organized as follows. In \Cref{sec:HRS}, we introduce the notion of $m$-extended heart of a bounded $t$-structure on a triangulated category and show a Happel-Reiten-Smal{\o} tilting theorem for an $s$-torsion pair in an extended heart. We also introduce the notion of $m$-factors and $m$-subobjects and use them to give a necessary and sufficient condition for a torsion pair to be an $s$-torsion pair. In \Cref{sec:tp}, we introduce the notion of $m$-extended module category of a finite-dimensional algebra and investigate the properties of torsion pairs within this category induced by $(m+1)$-term silting complexes. In \Cref{sec:AR}, we construct the Auslander-Reiten translations $\AR{m}$ and $\AR{m}^-$ and show that any $m$-extended module category has Auslander-Reiten triangles. In \Cref{sec:tau-tilting}, we develop the $\tau$-tilting theory for an extended module category. In \Cref{sec: trun}, we gather the basic constructions of the two methods used to truncate complexes in this paper.

\subsection*{Convention}

Throughout this paper, $\k$ is a field, and $m$ is a positive integer. Any category is assumed to be additive. Any subcategory of a category is assumed to be full. We use $\XX\in\X$ to denote that $\XX$ is an object in a category $\X$, and use $\X\subseteq\Y$ to denote that $\X$ is a subcategory of $\Y$. For any $\XX\in\X$, we denote by $\add\XX$ the additive hull of $\XX$, that is, the smallest additive subcategory of $\X$ containing $\XX$. For any two morphisms $f:\YY\to\ZZ$ and $g:\XX\to\YY$, we denote by $f\circ g:\XX\to\ZZ$ their composition. 

Let $\Y$ be a category. For any $\X\subseteq\Y$ and any $\YY\in\Y$, a right $\X$-approximation of $\YY$ is a morphism $f:\XX\to\YY$ with $\XX\in\X$ and such that any morphism from an object $\XX'$ in $\X$ to $\YY$ factors through $f$. A left $\X$-approximation of $\YY$ is defined dually. A subcategory $\X$ of $\Y$ is called contravariantly (resp. covariantly) finite in $\Y$ if any object in $\Y$ has a right (resp. left) $\X$-approximation. A subcategory $\X$ of $\Y$ is called functorially finite if it is both contravariantly finite and covariantly finite. A morphism $f:\XX\to\YY$ is called right minimal if every morphism $g:\XX\to\XX$ such that $f\circ g=f$ is an isomorphism. A left minimal morphism is defined dually.

The shift functor in a triangulated category is denoted by $[1]$. For any two subcategories $\X$ and $\Y$ of a triangulated category $\D$, we denote by $\X*\Y$ the subcategory of $\D$ consisting of objects $\ZZ$ such that there is a triangle
$$\XX\to\ZZ\to\YY\to\XX[1],$$
with $\XX\in\X$ and $\YY\in\Y$. We use $\Hom(\X,\Y)=0$ to denote $\Hom(\XX,\YY)=0$ for any $\XX\in\X$ and $\YY\in\Y$.

For a vertex $i$ of a quiver $Q$, we denote by $P_i$, $I_i$, and $S_i$, respectively, the corresponding projective, injective, and simple representations. For two arrows $a$ and $b$ of $Q$, we denote by $ab$ the path first $a$ then $b$.

\subsection*{Acknowledgments}

I would like to thank Aslak Bakke Buan, Xiao-Wu Chen, Esha Gupta, Peter J{\o}rgenson, Jiaqun Wei, Dong Yang, Zhaotai Zhang and Bin Zhu for their interesting and helpful discussions.

\section{Generalized Happel-Reiten-Smalo tilting}\label{sec:HRS}

Let $\mathcal{D}$ be a triangulated category. Throughout this section, for any two objects $\XX$ and $\YY$ in $\D$, we simply denote by $\Hom(\XX,\YY)=\Hom_\D(\XX,\YY)$ the set of morphisms from $\XX$ to $\YY$ in $\D$.

Let $(\mathcal{D}^{\leq 0},\mathcal{D}^{\geq 0})$ be a bounded $t$-structure on $\mathcal{D}$, i.e., $\mathcal{D}^{\leq 0}$ and $\mathcal{D}^{\geq 0}$ are subcategories of $\D$ satisfying the following.
\begin{enumerate}
    \item $\mathcal{D}^{\leq 0}[1]\subseteq\mathcal{D}^{\leq 0}$ and $\mathcal{D}^{\geq 0}[-1]\subseteq\mathcal{D}^{\geq 0}$,
    \item $\Hom(\mathcal{D}^{\leq 0},\mathcal{D}^{\geq 0}[-1])=0$,
    \item $\D=\mathcal{D}^{\leq 0}*\mathcal{D}^{\geq 0}[-1]$, and
    \item for any object $\XX$ in $\D$, there exists an integer $n>0$ such that $\XX[n]\in\mathcal{D}^{\leq 0}$ and $\XX[-n]\in\mathcal{D}^{\geq 0}$.
\end{enumerate}
Let $\mathcal{H}=\mathcal{D}^{\leq 0}\cap\mathcal{D}^{\geq 0}$ be its heart. 

\begin{notation}\label{trun1}
    For any integer $p$, we denote 
    $$\mathcal{D}^{\leq p}=\mathcal{D}^{\leq 0}[-p]\text{ and }\mathcal{D}^{\geq p}=\mathcal{D}^{\geq 0}[-p].$$
    For any integers $p\leq q$, we denote
    $$\mathcal{D}^{[p,q]}=\mathcal{D}^{\geq p}\cap\mathcal{D}^{\leq q}.$$
\end{notation}

For any integer $p$, the inclusions $\D^{\leq p}\to\D$ and $\D^{\geq p}\to\D$ admit adjoints
$$\sigma^{\leq p}:\D\to \D^{\leq p}\text{ and }\sigma^{\geq p}:\D\to \D^{\geq p},$$
which are called truncation functors.

\begin{notation}\label{trun2}
    For any integers $p\leq q$, we denote
    $$H^{[p,q]}=\sigma^{\geq p}\circ\sigma^{\leq q}\simeq\sigma^{\leq q}\circ\sigma^{\geq p}:\D\to \D^{[p,q]}.$$
    For any integer $p$, we denote
    $$H^p=H^{[p,p]}[p]:\D\to\mathcal{H}$$
    the $p$-th cohomology functor with respect to the $t$-structure $(\mathcal{D}^{\leq 0},\mathcal{D}^{\geq 0})$.
\end{notation}

\begin{definition}\label{def:exth}
    We call $\mathcal{D}^{[-(m-1),0]}$ the $m$-extended heart of a bounded $t$-structure $(\mathcal{D}^{\leq 0},\mathcal{D}^{\geq 0})$.
\end{definition}

By definition, the $1$-extended heart $\mathcal{D}^{[0,0]}$ is the (usual) heart $\mathcal{H}$ and in general,
$$\mathcal{D}^{[-(m-1),0]}=\mathcal{H}[m-1]*\cdots *\mathcal{H}[1]*\mathcal{H}.$$

\begin{remark}\label{extri}
    An extriangulated category is a triplet $(\mathcal{E},\mathbb{E},\mathfrak{s})$, where $\mathcal{E}$ is an additive category, $\mathbb{E}:\mathcal{E}^{\operatorname{op}}\times\mathcal{E}\to Ab$ is a biadditive functor, and $\mathfrak{s}$ is an additive realization of $\mathbb{E}$, satisfying certain conditions. We refer to \cite[Definition~2.12]{NP} for more details. Since $\mathcal{D}^{[-(m-1),0]}$ is closed under extensions (to see this, take the cohomologies), it becomes an extriangulated category, equipped with $\mathbb{E}(\XX,\YY):=\Hom(\XX,\YY[1])$ and $\mathfrak{s}(\delta)=[\YY\xrightarrow{y}\ZZ\xrightarrow{x}\XX]$ for $\delta\in\mathbb{E}(\XX,\YY)$, where $\YY\xrightarrow{y}\ZZ\xrightarrow{x}\XX\xrightarrow{\delta}\YY[1]$ is a triangle in $\D$. Moreover, $\mathcal{D}^{[-(m-1),0]}$ is an extriangulated category with a negative first extension in the sense of \cite[Definition~2.3]{AET} (cf. Example~2.4 there), where $\mathbb{E}^{-1}(\XX,\YY):=\Hom(\XX,\YY[-1])$.
\end{remark}

\begin{definition}
    A pair $(\T,\F)$ of subcategories of $\mathcal{D}^{[-(m-1),0]}$ is called a torsion pair provided that the following hold.
    \begin{enumerate}
    \item $\Hom(\T,\F)=0$, and
    \item $\mathcal{D}^{[-(m-1),0]}=\T\ast\F$.
    \end{enumerate}
\end{definition}

\begin{remark}\label{tor in extri}
The above definition of torsion pairs coincides with that in an extriangulated category introduced in \cite{HHZ}. However, this is not the case in \cite{G}, see \Cref{rmk:G1}. 
\end{remark}

\begin{definition}[{\cite[Definition~3.1]{AET}}]
    A torsion pair $(\T,\F)$ in $\D^{[-(m-1),0]}$ is called an $s$-torsion pair, if 
    $\Hom(\T,\F[-1])=0.$
\end{definition}

In the usual case (i.e., $m=1$), any torsion pair in $\D^{[0,0]}=\mathcal{H}$ is an $s$-torsion pair, because $\Hom(\mathcal{H},\mathcal{H}[-1])=0$.

For any subcategory $\X$ of $\D$, we denote by $\X^\bot$ (resp. ${}^\bot\X$) the subcategory of $\D$ consisting of objects $\ZZ$ satisfying $\Hom(\XX,\ZZ)=0$ (resp. $\Hom(\ZZ,\XX)=0$) for any $\XX\in\X$. The following lemma is from \cite[Proposition~3.2]{AET}.

\begin{lemma}\label{lem:summand}
    Let $(\T,\F)$ be an $s$-torsion pair in $\D^{[-(m-1),0]}$. Then $\T={}^\bot\F$ and $\F=\T^\bot$. In particular, $\T$ and $\F$ are closed under taking extensions and direct summands.
\end{lemma}

The notion of $s$-torsion pairs was introduced in \cite{AET} for arbitrary extriangulated categories with a negative first extension. In particular, $s$-torsion pairs in the triangulated category $\D$ are exactly $t$-structures. Hence, applying \cite[Theorem~3.9]{AET} to $\D^{[-(m-1),0]}$, we get the following bijection.

\begin{proposition}\label{prop:bi}
    There is a bijection between 
    \begin{itemize}
        \item the set of bounded $t$-structures $(\C^{\leq0},\C^{\geq 0})$ on $\D$, satisfying $\mathcal{D}^{\leq -m}\subseteq\C^{\leq0}\subseteq\mathcal{D}^{\leq 0}$, and
        \item  the set of $s$-torsion pairs in $\mathcal{D}^{[-(m-1),0]}$, 
    \end{itemize}
    by the map $$(\C^{\leq0},\C^{\geq 0})\mapsto(\C^{\leq0}\cap\mathcal{D}^{[-(m-1),0]},\C^{\geq 1}\cap \mathcal{D}^{[-(m-1),0]}),$$
    with inverse
    $$(\T,\F)\mapsto (\mathcal{D}^{\leq -m}*\T,\F[1]*\mathcal{D}^{\geq 0}).$$
\end{proposition}

The following result shows that the vanishing condition on the first negative extension in the definition of $s$-torsion pairs is equivalent to the vanishing condition on all negative extensions.

\begin{corollary}\label{cor:positive}
    Let $(\T,\F)$ be an $s$-torsion pair in $\D^{[-(m-1),0]}$. Then 
    \begin{equation}\label{eq:neg}
        \Hom(\T,\F[j])=0,\ j<0.
    \end{equation}
\end{corollary}

\begin{proof}
    By \Cref{prop:bi}, there is a bounded $t$-structure $(\C^{\leq 0},\C^{\geq 0})$ such that $\T=\C^{\leq 0}\cap\D^{[-(m-1),0]}$ and $\F=\C^{\geq 1}\cap\D^{[-(m-1),0]}$. Since $\C^{\geq 1}$ is closed under $[-1]$, we have $\F[j]\subseteq\C^{\geq 1}$ for any $j<0$. So we have $\Hom(\T,\F[j])=0$.
\end{proof}

\begin{remark}\label{rmk:G1}
    A torsion pair in the sense of \cite[Definition~3.14~(1)]{G} is a pair $(\T,\F)$ of subcategories satisfying $\T={}^\bot\F$ and $\F=\T^\bot$. A torsion pair in the sense of \cite{G} is called positive if the equality~\eqref{eq:neg} holds. So by \Cref{lem:summand} and \Cref{cor:positive}, an $s$-torsion pair is always a positive torsion pair. However, the converse is not true in general, see \cite[Example~5.3]{G}.
\end{remark}

The following is a generalization of Happel-Reiten-Smal{\o} tilting.

\begin{theorem}\label{prop:hrs}
    Let $(\T,\F)$ be an $s$-torsion pair in $\D^{[-(m-1),0]}$ and let $(\C^{\leq 0},\C^{\geq 0})$ be the corresponding bounded $t$-structure as in \Cref{prop:bi}. Then $\F[m]*\T$ is the $m$-extended heart $\C^{[-(m-1),0]}$ of $(\C^{\leq 0},\C^{\geq 0})$. In particular, $(\F[m],\T)$ is an $s$-torsion pair in $\C^{[-(m-1),0]}$.
\end{theorem}

\begin{proof}
    Since $\T,\F\subseteq\D^{[-(m-1),0]}$, we have $\T\subseteq \mathcal{D}^{\geq -(m-1)}$ and $\F[m]\subseteq\D^{\leq -m}$. So we have $\Hom(\F[m],\T)=0$. Thus, we only need to show $\F[m]*\T=\C^{[-(m-1),0]}$. On the one hand, we have
    $$\begin{array}{rcl}
       \F[m]*\T &\subseteq & \left(\mathcal{D}^{\leq -m}*\T\right)\cap\left(\F[m]*\mathcal{D}^{\geq -(m-1)}\right)\\
       & = & (\mathcal{D}^{\leq -m}*\T)\cap\left((\F[1]*\mathcal{D}^{\geq 0})[m-1]\right) \\
       & = & \C^{[-(m-1),0]}.
    \end{array}$$
    Conversely, for any object $\XX$ in $\C^{[-(m-1),0]}=(\mathcal{D}^{\leq -m}*\T)\cap(\F[m]*\mathcal{D}^{\geq -(m-1)})$, there are triangles
    \[\YY\xrightarrow{f}\XX\xrightarrow{g}\ZZ\to\YY[1]\]
    and
    \[\YY'\xrightarrow{h}\XX\to\ZZ'\to\YY'[1]\]
    with $\YY\in \mathcal{D}^{\leq -m}$, $\ZZ\in\T$, $\YY'\in\F[m]$ and $\ZZ'\in \mathcal{D}^{\geq -(m-1)}$. Since $\Hom(\F[m],\T)=0$, $g\circ h=0$. So $h$ factors through $f$. Hence, by the octahedral axiom, there is the following commutative diagram of triangles
    \[\xymatrix{
    &\YY'\ar@{=}[r]\ar[d]&\YY'\ar[d]^{h}\\
    \ZZ[-1]\ar[r]\ar@{=}[d]&\YY\ar[d]^{a}\ar[r]^{f}&\XX\ar[r]^{g}\ar[d]&\ZZ\ar@{=}[d]\\
    \ZZ[-1]\ar[r]&\YY''\ar[r]\ar[d]&\ZZ'\ar[d]\ar[r]&\ZZ\\
    &\YY'[1]\ar@{=}[r]&\YY'[1]
    }
    \]
    By the triangle in the third row of the diagram, we have $$\YY''\in \add\ZZ[-1]*\add\ZZ'\subseteq\T[-1]\ast\D^{\geq -(m-1)}\subseteq \D^{\geq -(m-1)}.$$
    So $\Hom(\YY,\YY'')=0$, which implies that the morphism $a:\YY\to\YY''$ in the above diagram is zero. It follows that $\YY$ is a direct summand of $\YY'$ and hence belongs to $\F[m]$ by \Cref{lem:summand}. Therefore, $\XX$ is an object in $\F[m]*\T$.
\end{proof}

To give an alternative description for a torsion pair to be an $s$-torsion pair, we introduce the notions of $n$-factors and $n$-subobjects, which will also play an important role in the study of torsion theory and $\tau$-tilting theory in $m$-extended module categories in the remaining sections. We refer to \cite{W1,W2} for similar notions for module categories.

\begin{definition}\label{def:sfac}
    Let $\X$ be a subcategory of $\D^{[-(m-1),0]}$ and $n$ be a positive integer. 
    \begin{enumerate}
    \item An object $\ZZ$ in $\D^{[-(m-1),0]}$ is called an $n$-factor of $\X$ provided that there are $n$ many triangles
    \begin{equation}\label{eq:fac0}
        \ZZ_{i}\to\XX_{i}\to\ZZ_{i-1}\to\ZZ_{i}[1],\ 1\leq i\leq n,
    \end{equation}
    with $\ZZ_0=\ZZ,\ZZ_1,\cdots,\ZZ_n\in\D^{[-(m-1),0]}$ and $\XX_1,\cdots,\XX_n\in\X$. Denote by $\Fac_n(\X)$ the subcategory of $\D^{[-(m-1),0]}$ consisting of all $n$-factors of $\X$.
    \item An object $\ZZ$ in $\D^{[-(m-1),0]}$ is called an $n$-subobjects of $\X$ provided that there are $n$ many triangles
    \begin{equation}\label{eq:sub0}
        \ZZ_{i-1}\to\XX_{i}\to\ZZ_{i}\to\ZZ_{i-1}[1],\ 1\leq i\leq n,
    \end{equation}
    with $\ZZ_0=\ZZ,\ZZ_1,\cdots,\ZZ_n\in\D^{[-(m-1),0]}$ and $\XX_1,\cdots,\XX_n\in\X$. Denote by $\Sub_n(\X)$ the subcategory of $\D^{[-(m-1),0]}$ consisting of all $n$-subobjects of $\X$.
    \end{enumerate}
\end{definition}

The notion of 1-factors (resp. 1-subobjects) coincides with the usual notion of factors (resp. subobjects) in an extriangulated category (at least when the extriangulated category is an abelian category).

\begin{example}\label{exm:fac}
    Let $\X$ be a subcategory of $\D^{[-(m-1),0]}$ and $n$ be a positive integer.
    \begin{enumerate}
        \item For any $\YY\in\D^{[-(m-1),0]}$, if $\YY[n]$ is also an object in $\D^{[-(m-1),0]}$, then $\YY[n]\in \Fac_{n}(\X)$. Note that $\YY[n]\in\D^{[-(m-1),0]}$ implies $\YY\in\D^{[-(m-1)+n,n]}$. So $\YY\in\D^{[-(m-1),0]}\cap\D^{[-(m-1)+n,n]}=\D^{[-(m-1)+n,0]}$. Hence, $\YY[i]\in\D^{[-(m-1),0]}$ for any $0\leq i\leq n$. Thus, the assertion follows directly from the triangles
        $$\YY[i-1]\to 0\to \YY[i]\to \YY[i],\ 1\leq i\leq n.$$
        \item Similarly, for any $\ZZ\in\Fac_{n}(\X)$ and any positive integer $l$, if $\ZZ[l]$ is also an object in $\D^{[-(m-1),0]}$, then $\ZZ[l]\in \Fac_{n+l}(\X)$.
    \end{enumerate}
\end{example}

\begin{remark}
    For the relationship between $n$-factors and $(n+1)$-factors, on the one hand, by the construction, we have
    \begin{equation}\label{eq:rec}
        \Fac_{n+1}(\X)=(\X\ast\Fac_{n}(\X)[1])\cap\D^{[-(m-1),0]}.
    \end{equation}
    On the other hand, one has 
    \begin{equation}\label{eq:s subset s-1}
        \Fac_{n+1}(\X)\subseteq\Fac_{n}(\X),
    \end{equation}
    because an object admitting $n+1$ many triangles in \eqref{eq:fac0} certainly admits $n$ many such triangles.
\end{remark}

The following easy observation is useful.

\begin{lemma}\label{lem:fac1}
    Let $\X$ be a subcategory of $\D^{[-(m-1),0]}$. Then for any $\XX\in\X$, we have $H^{[-(m-2),0]}(\XX)\in \Fac_{m}(\X)$.
\end{lemma}

\begin{proof}
    By \Cref{exm:fac}~(1), we have $H^{-(m-1)}(\XX)[m-1]\in \Fac_{m-1}(\X)$. Then by the following triangle given by truncation of $\XX$
    $$H^{-(m-1)}(\XX)[m-1]\to \XX\to H^{[-(m-2),0]}(\XX)\to H^{-(m-1)}(\XX)[m]$$
    and the equality~\eqref{eq:rec}, we have $H^{[-(m-2),0]}(\XX)\in \Fac_{m}(\X)$.
\end{proof}

\begin{definition}\label{closed}
    A subcategory $\X$ of $\D^{[-(m-1),0]}$ is called closed under $n$-factors (resp. $n$-subobjects) if $\Fac_{n}(\X)\subseteq\X$ (resp. $\Sub_{n}(\X)\subseteq\X$).
\end{definition}

We give an alternative description of $s$-torsion pairs.

\begin{proposition}\label{prop:fac}
    Let $(\T,\F)$ be a torsion pair in $\D^{[-(m-1),0]}$. The following are equivalent.
    \begin{enumerate}
        \item $(\T,\F)$ is an $s$-torsion pair.
        \item $\T$ is closed under $m$-factors.
        \item $\F$ is closed under $m$-subobjects.
    \end{enumerate}
\end{proposition}

\begin{proof}
    We only show the equivalence (1) $\Leftrightarrow$ (2) since the equivalence (1) $\Leftrightarrow$ (3) can be shown similarly.
    
    (1) $\Rightarrow$ (2): Let $\ZZ$ be an object in $\Fac_{m}(\T)$. Then, by definition, there are triangles
    \[\ZZ_{i}\to\TT_{i}\to\ZZ_{i-1}\to\ZZ_{i}[1],\ 1\leq i\leq m,\]
    with $\ZZ_0=\ZZ,\ZZ_1,\cdots,\ZZ_m\in\D^{[-(m-1),0]}$ and $\TT_1,\cdots,\TT_m\in\T$. For any $\FF\in\F$, applying $\Hom(-,\FF)$ to these triangles, we get exact sequences for all integers $j$:
    $$\Hom(\TT_i,\FF[j])\to\Hom(\ZZ_{i},\FF[j])\to\Hom(\ZZ_{i-1},\FF[j+1])\to \Hom(\TT_i,\FF[j+1]).$$
    When $j\leq -1$, since $\TT_i\in\T$ and $\FF\in\F$,  by \Cref{cor:positive}, we have $$\Hom(\TT_i,\FF[j])=0=\Hom(\TT_i,\FF[j+1]).$$
    Hence, there are isomorphisms
    \[\Hom(\ZZ_{i},\FF[j])\cong\Hom(\ZZ_{i-1},\FF[j+1]),\ j\leq -1.\]
    Due to $\ZZ_m\in\D^{[-(m-1),0]}$ and $\FF[-m]\in\F[-m]\subseteq\D^{\geq 1}$, we have $\Hom(\ZZ_m,\FF[-m])=0$. Hence, by recursion, one has $\Hom(\ZZ_0,\FF)=0$. This implies $\ZZ=\ZZ_0\in\T$ by \Cref{lem:summand}. Therefore, $\T$ is closed under $m$-factors.

    (2) $\Rightarrow$ (1): Let $\TT\in\T$ and $\FF\in\F$. Since $\TT[1]\in\T[1]\subseteq\D^{[-m,-1]}$, by truncation of $\TT[1]$, there is a triangle
    \begin{equation}\label{eq:linshi}
        \YY[m]\to \TT[1]\to \ZZ\to \YY[m+1],
    \end{equation}
    with $\YY=H^{-m}(\TT[1])$ and $\ZZ=H^{[-(m-1),-1]}(\TT[1])\in \D^{[-(m-1),-1]}\subseteq \D^{[-(m-1),0]}$.
    Applying $\Hom(-,\FF)$ to this triangle, we get an exact sequence
    $$\Hom(\YY[m+1],\FF)\to\Hom(\ZZ,\FF)\to\Hom(\TT[1],\FF)
        \to \Hom(\YY[m],\FF).$$
    Since $\FF\in\D^{\geq -(m-1)}$, the first item and the last item in the above sequence are zero. Hence, there is an isomorphism
    $$\Hom(\ZZ,\FF)\cong\Hom(\TT[1],\FF).$$
    By \Cref{exm:fac}~(1), $\YY[m-1]\in \Fac_{m-1}(\T)$. Shifting the triangle~\eqref{eq:linshi} by $[-1]$, we get a triangle
    $$\YY[m-1]\to \TT\to \ZZ[-1]\to \YY[m].$$
    Since $\ZZ[-1]\in \D^{[-(m-1),-1]}[-1]=\D^{[-(m-2),0]}\subseteq\D^{[-(m-1),0]}$, by the equality~\eqref{eq:rec}, we have $\ZZ[-1]\in \Fac_{m}(\T)$. So, by \Cref{exm:fac}~(2), $\ZZ\in \Fac_{m+1}(\T)$. Hence, by the inclusion~\eqref{eq:s subset s-1}, $\ZZ\in \Fac_{m}(\T)$. Then $\ZZ\in\T$, since $\T$ is closed under $m$-factors. Thus, we have $\Hom(\TT[1],\FF)\cong\Hom(\ZZ,\FF)=0$. Therefore, $(\T,\F)$ is an $s$-torsion pair.
\end{proof}

The following result about $m$-factors will be used later.

\begin{lemma}\label{lem:asfac}
    Let $\X$ be a subcategory of $\D^{[-(m-1),0]}$ and $\YY\in\D^{[-(m-1),0]}$. If $\Hom(\X,\YY[j])=0$ for any $j\leq 0$, then $\Hom(\Fac_{m}(\X),\YY[j])=0$ for any $j\leq 0$.
\end{lemma}

\begin{proof}
    For any $\ZZ\in\Fac_{m}(\X)$, by definition, there are triangles
    \[\ZZ_{i}\to\XX_{i}\to\ZZ_{i-1}\to\ZZ_{i}[1],\ 1\leq i\leq m,\]
    with $\ZZ_0=\ZZ,\ZZ_1,\cdots,\ZZ_m\in\D^{[-(m-1),0]}$ and $\XX_1,\cdots,\XX_m\in\X$. Applying the functor $\Hom(-,\YY)$ to these triangles, we get exact sequences for all integers $l$ and $1\leq i\leq m$
    \[\Hom(\XX_i,\YY[l])\to\Hom(\ZZ_{i},\YY[l])\to\Hom(\ZZ_{i-1},\YY[l+1])\to\Hom(\XX_i,\YY[l+1]).\]
    When $l\leq -1$, by the assumption, the first item and the last item in the above sequence are zero. Hence, there are isomorphisms
    $$\Hom(\ZZ_{i},\YY[l])\cong\Hom(\ZZ_{i-1},\YY[l+1]),\ 1\leq i\leq m,\ l\leq -1.$$
    Due to $\YY,\ZZ_m\in\D^{[-(m-1),0]}$, we have $\Hom(\ZZ_{m},\YY[-m][j])=0$ for any $j\leq 0$. Hence, by recursion, $\Hom(\ZZ_{i},\YY[j-i])=0$ for any $0\leq i\leq m$ and $j\leq 0$. In particular, for $i=0$, we have $\Hom(\ZZ,\YY[j])=0$ for any $j\leq 0$.
\end{proof}

\section{Torsion pairs induced by silting complexes}\label{sec:tp}

Throughout the rest of the paper, let $A$ be a finite-dimensional algebra over $\k$. We denote by 
\begin{itemize}
    \item $\mod A$ the category of finitely generated right $A$-modules,
    \item $\proj A$ (resp. $\inj A$) the subcategory of $\mod A$ consisting of projective (resp. injective) modules,
    \item $D^b(\mod A)$ the bounded derived category of $\mod A$,
    \item $K^b(\proj A)$ (resp. $K^b(\inj A)$) the bounded homotopy category of $\proj A$ (resp. $\inj A$).
\end{itemize}
We regard $\mod A$, $K^b(\proj A)$ and $K^b(\inj A)$ as subcategories of $D^b(\mod A)$ in a natural way. For any two objects $\XX$ and $\YY$ in $D^b(\mod A)$, we simply denote by $$\Hom(\XX,\YY)=\Hom_{D^b(\mod A)}(\XX,\YY)$$ the set of morphisms from $\XX$ to $\YY$ in $D^b(\mod A)$. 

Let $D=\Hom_\k(-,\k)$ be the standard $\k$-linear duality. Denote by
$$\nu=D\Hom_A(-,A):\proj A\to \inj A$$
the Nakayama functor, with quasi-inverse $\nu^{-}=\Hom_A(DA,-)$. They induce an equivalence
\begin{equation}\label{Naks}
    \nu:K^b(\proj A)\to K^b(\inj A)
\end{equation}
with quasi-inverse 
\begin{equation}\label{inNaks}
    \nu^{-}:K^b(\inj A)\to K^b(\proj A)
\end{equation}
There is a well-known duality, given as a functorial isomorphism
\begin{equation}\label{eq:Nakayama}
    \Hom(\XX,\YY)\cong D\Hom(\YY,\nu\XX).
\end{equation}
for any $\XX\in K^b(\proj A)$ and any $\YY\in D^b(\mod A)$. See e.g. \cite[Chapter~1, Section~4.6]{H}.

We denote by $(\D^{\leq 0},\D^{\geq 0})$ the canonical $t$-structure on $D^b(\mod A)$. That is,
$$\D^{\leq 0}:=\{\XX\in D^b(\mod A)\mid H^i(\XX)=0,\ \forall i>0\},$$
and
$$\D^{\geq 0}:=\{\XX\in D^b(\mod A)\mid H^i(\XX)=0,\ \forall i<0\}.$$
The heart $\D^{\leq 0}\cap\D^{\geq 0}$ can be identified with $\mod A$.

\begin{notation}
    For the canonical $t$-structure $(\D^{\leq 0},\D^{\geq 0})$, we continue using the concepts and notations introduced in \Cref{trun1} and \Cref{trun2}, where the truncation functors $\sigma^{\leq p}$ and $\sigma^{\geq p}$ are given by canonical truncation, cf. \Cref{sec: trun}.
\end{notation}

\begin{definition}
    The $m$-extended module category $\emod{m}{A}$ of $A$ is defined to be the subcategory of $D^b(\mod A)$ given as
    $$\begin{array}{rcl}
    \emod{m}{A} & := & \D^{\leq0}\cap \D^{\geq -(m-1)} \\
     & = & \{\XX\in D^b(\mod A)\mid H^i(\XX)=0,\ \forall i\notin [-(m-1),0]\}.
    \end{array}$$
\end{definition}

By definition, $\mod A=1\text{-}\mod A$. In general, we have
$$\emod{m}{A}=(\mod A)[m-1]\ast\cdots\ast(\mod A)[1]\ast\mod A.$$
Recall from \Cref{extri} that the extended module category $\emod{m}{A}$ is an extriangulated category with $\mathbb{E}(\XX,\YY)=\Hom(\XX,\YY[1])$ for any $\XX,\YY\in\emod{m}{A}$.

The $m$-extended module category $\emod{m}{A}$ is
\begin{itemize}
    \item $\k$-linear, i.e., $\Hom(\XX,\YY)$ is a vector space over $\k$ for any $\XX,\YY\in\emod{m}{A}$, and the compositions of morphisms are $\k$-linear,
    \item Hom-finite, i.e., $\dim_\k\Hom(\XX,\YY)<\infty$ for any $\XX,\YY\in\emod{m}{A}$,
    \item $\mathbb{E}$-finite, i.e., $\dim_\k\mathbb{E}(\XX,\YY)<\infty$ for any $\XX,\YY\in\emod{m}{A}$,
    \item Krull-Schmidt, i.e., any object $\XX$ in $\emod{m}{A}$ is isomorphic to a direct sum of objects whose local rings are local.
\end{itemize}
This is because the bounded derived category $D^b(\mod A)$ is $\k$-linear, Hom-finite and Krull-Schmidt.

Let $\PP$ be an $(m+1)$-term complex in $K^b(\proj A)$. That is, $\PP=(P^i,d^i:P^i\to P^{i+1})\in K^b(\proj A)$ with $P^i=0$ for any $i\notin [-m,0]$. By canonical truncation of $\PP$, there is a triangle
\begin{equation}\label{eq:tripp}
    H^{-m}(\PP)[m]\to \PP\to H^{[-(m-1),0]}(\PP)\to H^{-m}(\PP)[m+1],
\end{equation}
with $H^{[-(m-1),0]}(\PP)\in \emod{m}{A}$. 

The following easy observation is a generalization of \cite[Lemma~2.7]{BZ1}.

\begin{lemma}\label{lem:bz1-2.7}
    For any $\XX\in \emod{m}{A}$, there are functorial isomorphisms
    \begin{equation}\label{eq:iso}
    \Hom(\PP,\XX[j])\cong \Hom(H^{[-(m-1),0]}(\PP),\XX[j]),\ j\leq 0,
    \end{equation}
    and a monomorphism
    \begin{equation}\label{eq:mono}
        \Hom(H^{[-(m-1),0]}(\PP),\XX[1])\hookrightarrow \Hom(\PP,\XX[1]).
    \end{equation}
\end{lemma}

\begin{proof}
Applying $\Hom(-,\XX[j])$ to the triangle~\eqref{eq:tripp}, we obtain an exact sequence
\[
\begin{array}{rl}
     & \Hom(H^{m}(\PP)[m+1],\XX[j])\to \Hom(H^{[-(m-1),0]}(\PP),\XX[j])\to \Hom(\PP,\XX[j]) \\
    \to & \Hom(H^{-m}(\PP)[m],\XX[j]).
\end{array}
\]
Since $H^{-m}(\PP)[m+1]\in \D^{\leq -(m+1)}$, $H^{-m}(\PP)[m]\in \D^{\leq -m}$ and $\XX\in \D^{\geq-(m-1)}$, in the above sequence, the first item is zero for $j\leq 1$ and the last item is zero for $j\leq 0$. Thus, we get the required isomorphisms and monomorphism.
\end{proof}

Consider the following two subcategories of $\emod{m}{A}$:
\begin{equation}\label{eq:tpp}
    \T(\PP)=\{\XX\in\emod{m}{A}\mid\Hom(\PP,\XX[j])=0,\ 1\leq j\leq m\},
\end{equation}
and
\begin{equation}\label{eq:fpp}
    \F(\PP)=\{\XX\in\emod{m}{A}\mid\Hom(\PP,\XX[j])=0,\ -(m-1)\leq j\leq 0\}.
\end{equation}
Since $\PP$ is an $(m+1)$-term complex of projective modules, we have
\begin{equation}\label{eq:tpp1}
\T(\PP)=\{\XX\in\emod{m}{A}\mid\Hom(\PP,\XX[j])=0,\ j\geq 1\},
\end{equation}
and
\begin{equation}\label{eq:fpp1}
\F(\PP)=\{\XX\in\emod{m}{A}\mid\Hom(\PP,\XX[j])=0,\ j\leq 0\}.
\end{equation}

\begin{remark}
    Both $\T(\PP)$ and $\F(\PP)$ are closed under extensions. Therefore, they are extriangulated categories, with $\mathbb{E}(\XX,\YY)=\Hom(\XX,\YY[1])$.
\end{remark}

Recall from \Cref{def:sfac} that for any subcategory $\X$ of $\emod{m}{A}$, an $m$-factor of $\X$ is an object $\ZZ$ satisfying that there are triangles
\begin{equation}\label{eq:fac}
\ZZ_{i}\to\XX_i\to\ZZ_{i-1}\to\ZZ_{i}[1],\ 1\leq i\leq m,
\end{equation}
with $\ZZ_0=\ZZ,\ZZ_1,\cdots,\ZZ_m\in\emod{m}{A}$ and $\XX_1,\cdots,\XX_m\in\X$. The subcategory of $\emod{m}{A}$ consisting of all $m$-factors of $\X$ is denoted by $\Fac_m(\X)$. The subcategory $\X$ is called closed under $m$-factors if $\Fac_m(\X)\subseteq\X$. There are dual notions: $m$-subobjects, $\Sub_{m}(\XX)$ and closed under $m$-subobjects. 

When $\X=\add\XX$ for some object $\XX\in\emod{m}{A}$, we simply denote $\Fac_{m}(\XX)=\Fac_{m}(\add\XX)$ and $\Sub_{m}(\XX)=\Sub_{m}(\add\XX)$.

\begin{lemma}\label{lem:fac}
    The subcategory $\T(\PP)$ is closed under $m$-factors and the subcategory $\F(\PP)$ is closed under $m$-subobjects.
\end{lemma}

\begin{proof}
    We will show the first part of the assertion, as the argument for the second part is analogous. Let $\ZZ\in\Fac_{m}(\T(\PP))$. Then by definition, there are triangles~\eqref{eq:fac}, where $\X=\T(\PP)$. Applying $\Hom(\PP,-)$ to these triangles, we get exact sequences for all integers $j$ and $1\leq i\leq m$
    \[\Hom(\PP,\XX_{i}[j])\to\Hom(\PP,\ZZ_{i-1}[j])\to\Hom(\PP,\ZZ_{i}[j+1])\to\Hom(\PP,\XX_{i}[j+1]).\]
    When $j\geq 1$, by \eqref{eq:tpp1}, the first item and the last term in the above sequence are zero. Hence, there are isomorphisms
    $$\Hom(\PP,\ZZ_{i-1}[j])\cong\Hom(\PP,\ZZ_{i}[j+1]),\ 1\leq i\leq m,\ j\geq 1.$$
    For any $j\geq 1$, due to $\ZZ_m[j+m]\in\D^{\leq -(m+1)}$ and that $\PP$ is an $(m+1)$-term complex of projective modules, we have $\Hom(\PP,\ZZ_m[j+m])=0$. Hence, by recursion, $\Hom(\PP,\ZZ[j])=0$ for any $j\geq 1$. Thus, we have $\ZZ\in\T(\PP)$. So $\T(\PP)$ is closed under $m$-factors.
\end{proof}

A complex $\PP\in K^b(\proj A)$ is called \emph{presilting} if $\Hom(\PP,\PP[j])=0$ for any $j>0$. An object $\XX$ in an extriangulated category $\mathcal{E}$ is called \emph{projective} (resp. \emph{injective}) if $\mathbb{E}(\XX,\mathcal{E})=0$ (resp. $\mathbb{E}(\mathcal{E},\XX)=0$).

\begin{lemma}\label{lem:proj}
    Let $\PP$ be an $(m+1)$-term presilting complex in $K^b(\proj A)$. The following hold.
    \begin{enumerate}
        \item[(a)] $H^{[-(m-1),0]}(\PP)$ belongs to and is projective in $\T(\PP)$.
        \item[(b)] $H^{[-(m-1),0]}(\nu\PP[-1])$ belongs to and is injective in $\F(\PP)$.
    \end{enumerate}
\end{lemma}

\begin{proof}
    (a) Applying $\Hom(\PP,-)$ to the triangle~\eqref{eq:tripp}, we get exact sequences
    \[\Hom(\PP,\PP[i])\to \Hom(\PP,H^{[-(m-1),0]}(\PP)[i])\to\Hom(\PP,H^{-m}(\PP)[m+i+1]),\ i\in\mathbb{Z}.\]
    When $i\geq 1$, the first item is zero thanks to $\PP$ being presilting, and the last item is zero thanks to that $\PP$ is an $(m+1)$-term complex of projective modules. Therefore, $\Hom(\PP,H^{[-(m-1),0]}(\PP)[i])=0$ for any $i\geq 1$. Hence, $H^{[-(m-1),0]}(\PP)\in\T(\PP)$. 
    
    Next, for any $\TT\in\T(\PP)$, by \Cref{lem:bz1-2.7}, there is a monomorphism $$\Hom(H^{[-(m-1),0]}(\PP),\TT[1])\hookrightarrow\Hom(\PP,\TT[1])=0.$$
    Hence $\mathbb{E}(H^{[-(m-1),0]}(\PP),\TT)=\Hom(H^{[-(m-1),0]}(\PP),\TT[1])=0$, which implies that the object $H^{[-(m-1),0]}(\PP)$ is projective in $\T(\PP)$.

    (b) Let $\II=\nu\PP[-1]$. By canonical truncation of $\II$, there is a triangle
    \begin{equation}\label{eq:nuP}
    H^{[-(m-1),0]}(\II)\to \II \to H^1(\II)[-1]\to H^{[-(m-1),0]}(\II)[1].
    \end{equation}
    Applying $\Hom(\PP-)$ to this triangle, we get exact sequences
    $$\Hom(\PP,H^1(\II)[i-2])\to\Hom(\PP,H^{[-(m-1),0]}(\II)[i])\to\Hom(\PP,\II[i]),\ i\in\mathbb{Z}.$$
    When $i\leq 0$, the first item is zero because $\PP\in\D^{\leq 0}$ and $H^1(\II)[i-2]\in\D^{\geq 2}$, and the last item $\Hom(\PP,\II[i])=\Hom(\PP,\nu\PP[i-1])$, by the duality \eqref{eq:Nakayama}, is isomorphic to $D\Hom(\PP[i-1],\PP)=0$, since $\PP$ is presilting. Hence, $\Hom(\PP,H^{[-(m-1),0]}(\II)[i])=0$ for any $i\leq 0$. Thus, $H^{[-(m-1),0]}(\nu\PP[-1])$ belongs to $\F(\PP)$. 
    
    Next, for any $\FF\in\F(\PP)$, applying $\Hom(\FF,-)$ to the triangle~\eqref{eq:nuP}, we get an exact sequence
    \[\Hom(\FF,H^1(\II)[-1])\to\Hom(\FF,H^{[-(m-1),0]}(\II)[1])\to\Hom(\FF,\II[1]),\]
    where the first item is zero because $\FF\in\D^{\leq 0}$ and $H^1(\II)[-1]\in\D^{\geq 1}$, and the last item $\Hom(\FF,\II[1])=\Hom(\FF,\nu\PP)$, which by the duality \eqref{eq:Nakayama}, is isomorphic to $D\Hom(\PP,\FF)=0$. So $\mathbb{E}(\FF,H^{[-(m-1),0]}(\II))=\Hom(\FF,H^{[-(m-1),0]}(\II)[1])=0$, which implies that the object $H^{[-(m-1),0]}(\nu\PP[-1])$ is injective in $\F(\PP)$.
\end{proof}

Let $\PP$ be an $(m+1)$-term presilting complex in $K^b(\proj A)$. Consider the following two subcategories of $D^b(\mod A)$:
$$D^{\leq 0}(\PP)=\{\mathbf{X}\in D^b(\mod A)\mid \Hom(\PP,\mathbf{X}[i])=0,\ \forall i>0\},$$
and
$$D^{\geq 0}(\PP)=\{\mathbf{X}\in D^b(\mod A)\mid \Hom(\PP,\mathbf{X}[i])=0,\ \forall i<0\}.$$

\begin{remark}\label{rmk:inter1}
    For any $\XX\in\D^{\leq -m}$, we have $\Hom(\PP,\XX[i])=0$ for any $i>0$, because $\PP$ is an $(m+1)$-term silting complex of projective modules and $\XX[i]\in\D^{\leq -(m+1)}$ for $i>0$. Hence,
    \begin{equation}\label{APP}
        \D^{\leq -m}\subseteq D^{\leq 0}(\PP).
    \end{equation}
    For any $\XX\in\D^{\geq 0}$, we have $\Hom(\PP,\XX[i])=0$ for any $i<0$, because $\PP\in\D^{\leq 0}$ and $\XX[i]\in\D^{\geq 1}$ for $i<0$. Hence,
    \begin{equation}\label{APP2}
        \D^{\geq 0}\subseteq D^{\geq 0}(\PP).
    \end{equation}
\end{remark}

A presilting complex $\PP\in K^b(\proj A)$ is called \emph{silting} if $\thick\PP=K^b(\proj A)$, where $\thick\PP$ denotes the smallest triangulated subcategory of $K^b(\proj A)$ containing $\PP$ and closed under taking direct summands. We provide a sufficient condition for an $(m+1)$-term presilting complex to be silting (which is also a necessary condition, see \Cref{prop:tpp1} below). This will be used in \Cref{sec:tau-tilting}.

\begin{proposition}\label{new}
    Let $\PP$ be an $(m+1)$-term presilting complex in $K^b(\proj A)$. Suppose that for any $\XX\in\T(\PP)$, there is a triangle
    \[\ZZ\to\TT_0\to\XX\to\ZZ[1],\]
    with $\TT_0\in\add H^{[-(m-1),0]}(\PP)$ and $\ZZ\in\emod{m}{A}$. Then $\PP$ is silting.
\end{proposition}

\begin{proof}
    We first claim that
    \begin{equation}\label{eq:=pp}
        D^{\leq 0}(\PP)\subseteq\D^{\leq 0}.
    \end{equation}
    Suppose, to the contrary, that there is an object $\XX\in D^{\leq 0}(\PP)$ which is not in $\D^{\leq 0}$. Then there exists a positive integer $\xi>0$ such that $H^{\xi}(\XX)\neq 0$ and $H^{j}(\XX)=0$ for any $j>\xi$. By canonical truncation of $\XX$, there is a triangle
    $$\sigma^{\leq -m+\xi}(\XX)\to\XX\to H^{[-(m-1)+\xi,\xi]}(\XX)\to \sigma^{\leq -m+\xi}(\XX)[1].$$
    Applying $\Hom(\PP,-)$ to this triangle, we get an exact sequence for any integer $i$,
    $$\Hom(\PP,\XX[i])\to \Hom(\PP,H^{[-(m-1)+\xi,\xi]}(\XX)[i])\to \Hom(\PP,\sigma^{\leq -m+\xi}(\XX)[i+1]).$$
    By $\XX\in D^{\leq 0}(\PP)$, the first item is zero for any $i\geq 1$. Because $\PP$ is an $(m+1)$-term complex of projective modules and $\sigma^{\leq -m+\xi}(\XX)[i+1]\in \D^{\leq -m+\xi-i-1}$, the last item is zero for any $i\geq \xi$. Hence, $\Hom(\PP,H^{[-(m-1)+\xi,\xi]}(\XX)[i])=0$ for any $i\geq\xi$. Let $\MM=H^{[-(m-1)+\xi,\xi]}(\XX)[\xi]\in\emod{m}{A}$. Then $\Hom(\PP,\MM[i])=0$ for any $i\geq 0$. So $\MM\in\T(\PP)$. By \Cref{lem:bz1-2.7}, we have $\Hom(H^{[-(m-1),0]}(\PP),\MM)=0$. By the assumption of the proposition, there is a triangle
    \[\ZZ\to\TT_0\to\MM\to\ZZ[1],\]
    with $\TT_0\in\add H^{[-(m-1),0]}(\PP)$ and $\ZZ\in\emod{m}{A}$. Since the morphism from $\TT_0$ to $\MM$ is zero, it follows that $\MM[-1]$ is direct summand of $\ZZ$ and hence is in $\emod{m}{A}$. So $H^0(\MM)=H^{1}(\MM[-1])=0$, which implies $H^{\xi}(\XX)=H^{\xi}(H^{[-(m-1)+\xi,\xi]}(\XX))=H^0(\MM)=0$, a contradiction. Thus, we finish the proof of the claim \eqref{eq:=pp}.

    Next, for any $\XX\in D^{\leq 0}(\PP)$, take a right $\add\PP$-approximation $f$ of $\XX$ and extend it to a triangle in $D^b(\mod A)$
    $$\XX'\to \PP'\xrightarrow{f}\XX\to\XX'[1].$$
    Applying $\Hom(\PP,-)$ to it, there is a long exact sequence
    $$\begin{array}{cccccc}
         & \Hom(\PP,\PP')&\xrightarrow{\Hom(\PP,f)}&\Hom(\PP,\XX)&\to&\Hom(\PP,\XX'[1]) \\
         \to & \Hom(\PP,\PP'[1]) & \to & \cdots & \to & \cdots \\
         \to & \cdots & \to & \Hom(\PP,\XX[i]) & \to & \Hom(\PP,\XX'[i+1])\\
         \to & \Hom(\PP,\PP'[i+1]) & \to & \cdots
    \end{array}$$
    Since $f$ is a right $\add\PP$-approximation of $\XX$, the map $\Hom(\PP,f)$ is surjective. Since $\PP$ is presilting, $\Hom(\PP,\PP'[i])=0$ for any $i> 0$. Since $\XX\in D^{\leq 0}(\PP)$, we have $\Hom(\PP,\XX[i])=0$ for any $i> 0$. Hence, we have $\Hom(\PP,\XX'[i])=0$ for any $i>0$, which implies $\XX'\in D^{\leq 0}(\PP)$. Therefore, $D^{\leq 0}(\PP)\subseteq\add \PP* D^{\leq 0}(\PP)[1]$. Recursively, $D^{\leq 0}(\PP)\subseteq\add \PP*\add\PP[1]*\cdots*\add\PP[s]* D^{\leq 0}(\PP)[s+1]$ holds for any $s\geq 0$.
    
    Finally, to show $\PP$ is silting, it suffices to show $A[m]\in\thick(\PP)$.  Since $A[m]\in\D^{\leq -m}$, by \eqref{APP}, we have $A[m]\in D^{\leq 0}(\PP)\subseteq\add \PP*\add\PP[1]*\cdots*\add\PP[m]* D^{\leq 0}(\PP)[m+1]$. However, since by \eqref{eq:=pp}, $D^{\leq 0}(\PP)[m+1]\subseteq \D^{\leq -(m+1)}$, there is no non-zero morphism from $A[m]$ to any object in $D^{\leq 0}(\PP)[m+1]$. Thus, $A[m]$ is a direct summand of an object in $\add \PP*\add\PP[1]*\cdots*\add\PP[m]$. Therefore, $A[m]\in\thick(\PP)$.
\end{proof}

Throughout the rest of this section, let $\PP$ be an $(m+1)$-term silting complex. By \cite[Lemma~5.3]{KY}, the pair $(D^{\leq 0}(\PP),D^{\geq 0}(\PP))$ is a bounded $t$-structure on $D^b(\mod A)$ whose heart
$$\H(\PP)=\{\mathbf{X}\in D^b(\mod A)\mid \Hom(\PP,\mathbf{X}[i])=0,\ \forall i\neq 0\}$$
is equivalent to $\mod\End(\PP)$ by the functor $\Hom(\PP,-)$. 

We denote by $m\text{-}\H(\PP)$ the $m$-extended heart (see \Cref{def:exth}) of the $t$-structure $(D^{\leq 0}(\PP),D^{\geq 0}(\PP))$, that is,
$$\begin{array}{rcl}
   m\text{-}\H(\PP)  & = &  D^{\leq 0}(\PP)\cap D^{\geq -(m-1)}(\PP)\\
     & = & \{\mathbf{X}\in D^b(\mod A)\mid \Hom(\PP,\mathbf{X}[i])=0,\ \forall i\notin[-(m-1),0]\}.
\end{array}$$
For example, $A$ is an $(m+1)$-term silting complex in $K^b(\proj A)$, $D^{\leq 0}(A)=\D^{\leq 0}$, $D^{\geq 0}(A)=\D^{\geq 0}$ and $m\text{-}\H(A)=\emod{m}{A}$. 

\begin{remark}\label{rmk:interval}
    Since both $(\D^{\leq 0},\D^{\geq 0})$ and $(D^{\leq 0}(\PP),D^{\geq 0}(\PP))$ are t-structures, by \eqref{APP2}, we have 
    $$D^{\leq 0}(\PP)={}^\bot(D^{\geq 0}(\PP)[-1])\subseteq {}^\bot(\D^{\geq 0}[-1])=\D^{\leq 0}.$$
    In summary, we have
    $$\D^{\leq -m}\subseteq D^{\leq 0}(\PP)\subseteq \D^{\leq 0}.$$
\end{remark}

By \eqref{eq:tpp1} and \eqref{eq:fpp1}, we have
\[\T(\PP)=D^{\leq 0}(\PP)\cap \emod{m}{A}
\text{ and }
\F(\PP)=D^{\geq 1}(\PP)\cap \emod{m}{A}.
\]
Then we have the following direct application of \Cref{prop:bi} to the bounded derived category $D^b(\mod A)$. 

\begin{proposition}\label{lem:torA}
    $(\T(\PP),\F(\PP))$ is an $s$-torsion pair in $\emod{m}{A}$.
\end{proposition}

In the usual case (i.e., $m=1$), $(\mathcal{T}(\PP),\mathcal{F}(\PP))$ is an ($s$-)torsion pair in $\mod A$. A result stronger than \Cref{lem:torA} is shown in \cite{G}. See \Cref{rmk:ff} for a further discussion.

\begin{theorem}[{\cite[Theorem~4.1]{G}}]\label{rmk:G}
    The correspondence $\PP\mapsto(\T(\PP),\F(\PP))$ is a bijection from the set of (isoclasses of) basic $(m+1)$-term silting complexes in $K^b(\proj A)$ to the set of functorially finite $s$-torsion pairs in $\emod{m}{A}$.
\end{theorem}

It is also shown in \cite[Proposition~4.8]{G} that functorially finite $s$-torsion pairs are exactly functorially finite positive torsion pairs introduced there.

\begin{notation-remark}\label{tf}
    For any object $\XX$ in $\emod{m}{A}$, we denote by $\mathfrak{t}(\XX)\in\T(\PP)$ and $\mathfrak{f}(\XX)\in\F(\PP)$ the objects that fit into the following triangle
    $$\mathfrak{t}(\XX)\to\XX\to \mathfrak{f}(\XX)\to \mathfrak{t}(\XX)[1].$$
    This triangle is unique up to isomorphism (\cite[Proposition~3.7]{AET}), and is called the canonical triangle of $\XX$ with respect to the $s$-torsion pair $(\T(\PP),\F(\PP))$. In particular, the morphism $\mathfrak{t}(\XX)\to\XX$ (resp. $\XX\to \mathfrak{f}(\XX)$) in the triangle is a right (resp. left) minimal $\T(\PP)$-approximation ($\F(\PP)$-approximation) of $\XX$. Here, the minimality follows from the fact that otherwise, there is a nonzero direct summand $\YY$ of $\mathfrak{t}(\XX)$ whose shift $\YY[1]$ is a direct summand of $\mathfrak{f}(\XX)$, a contradiction with $\Hom(\T(\PP),\F(\PP)[-1])=0$.
\end{notation-remark}

\begin{lemma}\label{lem:extinj}
    The object $\mathfrak{t}(\nu A[m-1])$ is injective in $\T(\PP)$. The object $\mathfrak{f}(A)$ is projective in $\F(\PP)$.
\end{lemma}

\begin{proof}
    We only show the first assertion because the second one can be proved similarly. Take the canonical triangle of $\nu A[m-1]$ with respect to $(\T(\PP),\F(\PP))$
    \[\mathfrak{t}(\nu A[m-1])\to \nu A[m-1]\to \mathfrak{f}(\nu A[m-1])\to (\mathfrak{t}(\nu A[m-1]))[1].\]
    For any $\TT\in\T(\PP)$, applying $\Hom(\TT,-)$ to the triangle, we have an exact sequence
    \[\Hom(\TT,\mathfrak{f}(\nu A[m-1]))\to \Hom(\TT,\mathfrak{t}(\nu A[m-1])[1])\to \Hom(\TT,\nu A[m-1][1]),\]
    where the first item is zero thanks to $\mathfrak{f}(\nu A[m-1]))\in\F(\PP)$, and the last item $\Hom(\TT,\nu A[m])$, by the duality~\eqref{eq:Nakayama}, is isomorphic to $D\Hom(A[m],\TT)=0$, due to $A[m]\in\D^{\leq -m}$ and $\TT\in \T(\PP)\subseteq \D^{[-(m-1),0]}$. So we have $\mathbb{E}(\TT,\mathfrak{t}(\nu A[m-1]))=\Hom(\TT,\mathfrak{t}(\nu A[m-1])[1])=0$. Then $\mathfrak{t}(\nu A[m-1])$ is injective in $\T(\PP)$.
\end{proof}

The following proposition tells us that both of the extriangulated categories $\T(\PP)$ and $\F(\PP)$ contain enough projective objects and enough injective objects.

\begin{proposition}\label{prop:tpp1}
    Let $\PP$ be an $(m+1)$-term silting complex in $K^b(\proj A)$. The following hold.
    \begin{enumerate}
    \item[(a)] For any $\XX\in\T(\PP)$, there is a triangle
    \[\ZZ\to\TT_0\to\XX\to\ZZ[1],\]
    with $\TT_0\in\add H^{[-(m-1),0]}(\PP)$ and $\ZZ\in\T(\PP)$.
    \item[(b)] For any $\XX\in\T(\PP)$, $\XX$ is projective in $\T(\PP)$ if and only if $\XX$ belongs to $\add H^{[-(m-1),0]}(\PP)$.
    \item[(c)] For any $\XX\in \T(\PP)$, there is a triangle
    \[\XX\to\TT_0\to \ZZ\to\XX[1],\]
    with $\TT_0\in\add \mathfrak{t}(\nu A[m-1])$ and $\ZZ\in\T(\PP)$.
    \item[(d)] For any $\XX\in\T(\PP)$, $\XX$ is injective in $\T(\PP)$ if and only if $\XX$ belongs to $\add \mathfrak{t}(\nu A[m-1])$.
    \item[(e)] For any $\XX\in\F(\PP)$, there is a triangle
    \[\XX\to\FF_0\to\ZZ\to\XX[1],\]
    with $\FF_0\in\add H^{[-(m-1),0]}(\nu\PP[-1])$ and $\ZZ\in\F(\PP)$.
    \item[(f)] For any $\XX\in\F(\PP)$, $\XX$ is injective in $\F(\PP)$ if and only if $\XX$ belongs to $\add H^{[-(m-1),0]}(\nu\PP[-1])$.
    \item[(g)] For any $\XX\in\F(\PP)$, there is a triangle
    \[\ZZ\to\FF_0\to\XX\to\ZZ[1],\]
    with $\FF_0\in\add \mathfrak{f}(A)$ and $\ZZ\in\F(\PP)$.
    \item[(h)] For any $\XX\in\F(\PP)$, $\XX$ is projective in $\F(\PP)$ if and only if $\XX$ belongs to $\add \mathfrak{f}(A)$.
    \end{enumerate}
\end{proposition}

\begin{proof}
    We only prove (a)-(d) since (e)-(h) can be proved dually.

    (a) For any $\XX\in\T(\PP)$, take a right $\add\PP$-approximation $f_{\XX}:\PP_\XX\to\XX$ of $\XX$, where $\PP_\XX\in\add\PP$. Extended $f_{\XX}$ to a triangle in $D^b(\mod A)$:
    $$\YY\to\PP_{\XX}\xrightarrow{f_{\XX}} \XX\to\YY[1].$$
    Applying $\Hom(\PP,-)$ to this triangle, we get a long exact sequence
    $$\begin{array}{cccccc}
        &\Hom(\PP,\PP_{\XX}) & \xrightarrow{\Hom(\PP,f_\XX)} & \Hom(\PP,\XX) & \to & \Hom(\PP,\YY[1])  \\
        \to & \Hom(\PP,\PP_{\XX}[1]) & \to & \cdots & \to & \cdots \\
        \to & \cdots & \to & \Hom(\PP,\XX[i]) & \to & \Hom(\PP,\YY[i+1])  \\
        \to & \Hom(\PP,\PP_{\XX}[i+1]) & \to & \cdots & \\
    \end{array}$$
    Since $\PP$ is silting, we have $\Hom(\PP,\PP_{\XX}[i])=0$ for any $i\geq 1$. Since $\XX\in\T(\PP)$, by \eqref{eq:tpp1}, we have $\Hom(\PP,\XX[i])=0$ for any $i\geq 1$. Since $f_\XX$ is a right $\add\PP$-approximation of $\XX$, the map $\Hom(\PP,f_\XX)$ is surjective. Hence $\Hom(\PP,\YY[i])=0$ for any $i\geq 1$. So $\YY\in D^{\leq 0}(\PP)$.

    By canonical truncation of $\PP_\XX$, there is a triangle
    $$H^{-m}(\PP_\XX)[m]\to\PP_\XX\xrightarrow{g} H^{[-(m-1),0]}(\PP_\XX)\to H^{-m}(\PP_\XX)[m+1].$$
    Since due to $H^{-m}(\PP_\XX)[m]\in\D^{\leq -m}$ and $\XX\in\D^{[-(m-1),0]}$, there is no nonzero morphism from $H^{-m}(\PP_\XX)[m]$ to $\XX$, the morphism $f_\XX:\PP_\XX\to\XX$ factors through $g$. Hence, by the octahedral axiom, we have the following commutative diagram of triangles
    \[
    \xymatrix{
    &H^{-m}(\PP_{\XX})[m]\ar[d]\ar@{=}[r]&H^{-m}(\PP_{\XX})[m]\ar[d]\\
    \XX[-1]\ar[r]\ar@{=}[d]&\YY\ar[r]\ar[d]&\PP_{\XX}\ar[r]^{f_{\XX}}\ar[d]^{g}&\XX\ar@{=}[d]\\
    \XX[-1]\ar[r]&\ZZ\ar[r]\ar[d]&H^{[-(m-1),0]}(\PP_{\XX})\ar[r]\ar[d]&\XX\\
    &H^{-m}(\PP_{\XX})[m+1]\ar@{=}[r]&H^{-m}(\PP_{\XX})[m+1]
    }
    \]
    By the triangle in the third row, we have 
    $$\ZZ\in\add\XX[-1]*\add H^{[-(m-1),0]}(\PP_\XX) \subseteq\D^{[-(m-1),0]}[-1]*\D^{[-(m-1),0]}\subseteq \D^{[-(m-1),1]}.$$ 
    Since $H^{-m}(\PP_\XX)[m+1]\in\D^{\leq -(m+1)}\subseteq D^{\leq -1}(\PP)$, by the triangle in the second column of the above diagram, we have
    $$\ZZ\in\add\YY*\add H^{-m}(\PP_\XX)[m+1]\subseteq D^{\leq 0}(\PP)*D^{\leq -1}(\PP)\subseteq D^{\leq 0}(\PP)\subseteq \D^{\leq 0}.$$ 
    Hence, we have $\ZZ\in \D^{[-(m-1),0]}\cap D^{\leq 0}(\PP)=\T(\PP)$.

    (b) By \Cref{lem:proj}~(a), $H^{[-(m-1),0]}(\PP)$ is projective in $\T(\PP)$. So we only need to show that any projective object $\XX$ in $\T(\PP)$ belongs to $\add H^{[-(m-1),0]}(\PP)$. By (a), there is a triangle
    \[\ZZ\to \TT_0\to \XX\to\ZZ[1],\]
    with $\ZZ\in\T(\PP)$ and $\TT_0\in\add H^{[-(m-1),0]}(\PP)$. Since $\XX$ is projective in $\T(\PP)$, this triangle splits. Hence $\XX$ is a direct summand of $\TT_0$. So $\XX\in\add H^{[-(m-1),0]}(\PP)$.

    (c) Let $\alpha:\XX\to I_0[m-1]$ be a left $\add\nu A[m-1]$-approximation of $\XX$, where $I_0\in\add\nu A$. Take the canonical triangle of $I_0[m-1]$ with respect to the torsion pair $(\T(\PP),\F(\PP))$
    $$\mathfrak{t}(I_0[m-1])\xrightarrow{\beta} I_0[m-1]\to \mathfrak{f}(I_0[m-1])\to \mathfrak{t}(I_0[m-1])[1].$$
    Since $\XX\in\T(\PP)$, the morphism $\alpha$ factors through $\beta$. Then, by the octahedral axiom, we have the following commutative diagram of triangles
    \[\xymatrix{
    &\XX\ar@{=}[r]\ar[d]&\XX\ar[d]^{\alpha}\\
    \mathfrak{f}(I_0[m-1])[-1]\ar[r]\ar@{=}[d]&\mathfrak{t}(I_0[m-1])\ar[r]^{\beta}\ar[d]&I_0[m-1]\ar[r]\ar[d]&f(I_0[m-1])\ar@{=}[d]\\
    \mathfrak{f}(I_0[m-1])[-1]\ar[r]&\ZZ\ar[r]\ar[d]&\YY\ar[r]\ar[d]&\mathfrak{f}(I_0[m-1])\\
    &\XX[1]\ar@{=}[r]&\XX[1]
    }
    \]
    By the triangle in the second column, it suffices to show $\ZZ\in\T(\PP)$. 
    Applying $\Hom(-,\nu A)$ to the triangle in the third column of the above diagram, we
    get an exact sequence
    \[
    \begin{array}{rl}
         & \Hom(I_0[m-1],\nu A[m-1])\xrightarrow{\Hom(\alpha,\nu A[m-1])}\Hom(\XX,\nu A[m-1]) \\
        \to & \Hom(\YY,\nu A[m])  \to\Hom(I_0[m-1],\nu A[m])
    \end{array}
    \]
    and exact sequences
    \[
    \Hom(\XX,\nu A[m+i])\to \Hom(\YY,\nu A[m+i+1])\to \Hom(I_0[m-1],\nu A[m+i+1]),\ i\geq 0.
    \]
    Note that the map $\Hom(\alpha,\nu A[m-1])$ is surjective, since $\alpha$ is a left $\add\nu A[m-1]$-approximation of $\XX$. Note also that by the duality~\eqref{eq:Nakayama}, for any $i\geq 0$, we have
    $$\Hom(I_0[m-1],\nu A[m+i])\cong D\Hom(A[m+i],I_0[m-1])\cong D\Hom(A[i+1],I_0)=0$$ and 
    $$\Hom(\XX,\nu A[m+i])\cong D\Hom(A[m+i],\XX)=0.$$
    Hence, we have $\Hom(\YY,\nu A[m+i])=0$ for any $i\geq 0$. So, by the duality~\eqref{eq:Nakayama}, $\Hom(A[m+i],\YY)\cong D\Hom(\YY,\nu A[m+i])=0$ for any $i\geq 0$, which implies $\YY\in \D^{\geq -(m-1)}$. Then, by the triangle in the third row of the above diagram, we have 
    $$\ZZ\in \add \mathfrak{f}(I_0[m-1])[-1] * \add\YY \subseteq \D^{[-(m-1),0]}[-1]\ast \D^{\geq -(m-1)}\subseteq \D^{\geq -(m-1)}.$$
    On the other hand, since both $\XX$ and $\mathfrak{t}(I_0[m-1])$ belong to $\T(\PP)\subseteq D^{\leq 0}(\PP)$ and since $D^{\leq 0}(\PP)$ is closed under taking $[1]$, by the triangle in the second column of the above diagram, we have $\ZZ\in D^{\leq 0}(\PP)\subseteq\D^{\leq 0}$. Therefore, we have $\ZZ\in D^{\leq 0}(\PP)\cap\D^{\leq 0}\cap\D^{\geq -(m-1)}=\T(\PP)$.

    (d) Using \Cref{lem:extinj} and (c), the proof is similar to that of (b).
\end{proof}

We refer to \cite[Proposition~2.8]{BZ1} for the usual case (i.e., $m=1$) of the above properties of $\T(\PP)$ and $\F(\PP)$.

Now we are ready to show the main result in this section.

\begin{theorem}\label{prop:genTP}
    Let $\PP$ be an $(m+1)$-term silting complex in $K^b(\proj A)$. Then
    $$\T(\PP)=\Fac_{m}\left(H^{[-(m-1),0]}(\PP)\right)\text{ and }\F(\PP)=\Sub_{m}\left(H^{[-(m-1),0]}(\nu\PP[-1])\right).$$
\end{theorem}

\begin{proof}
    We only show the first equality since the second one can be proved similarly. Using \Cref{prop:tpp1}~(a) repeatedly, we have $\T(\PP)\subseteq\Fac_{m}(H^{[-(m-1),0]}(\PP))$. The converse inclusion follows from \Cref{lem:proj} and \Cref{lem:fac}.
\end{proof}

We refer to \Cref{ex:AR} for an example of the $s$-torsion pair $(\T(\PP),\F(\PP))$ induced by a silting complex $\PP$.

\begin{remark}\label{rmk:ff}
    By \Cref{rmk:G} and \Cref{prop:genTP}, for any $s$-torsion pair $(\T,\F)$, if it is functorially finite then $\T=\Fac_{m}(\XX)$ for some object $\XX$ in $\emod{m}{A}$. In the usual case (i.e., $m=1$), the converse is also true, see \cite{S}.
\end{remark}

\section{Auslander-Reiten theory in extended module categories}\label{sec:AR}

In this section, we study the Auslander-Reiten theory in the $m$-extended module category $\emod{m}{A}$. Recall that the category $\emod{m}{A}$ is an extriangulated category with
$\mathbb{E}(\XX,\YY)=\Hom(\XX,\YY[1])$. 

\begin{lemma}\label{lem:enough}
    The extended module category $\emod{m}{A}$ has enough projective objects $\proj A$ and enough injective objects $(\inj A)[m-1]$.
\end{lemma}

\begin{proof}
    For the $(m+1)$-term silting complex $A$, $\T(A)=\emod{m}{A}$, $H^{[-(m-1),0]}(A)=A$ and $\mathfrak{t}(\nu A[m-1])=\nu A[m-1]$. Hence, this assertion follows directly from \Cref{prop:tpp1}.
\end{proof}

\begin{notation}
    We denote by $\uemod{m}{A}$ (resp. $\oemod{m}{A}$) the quotient category of $\emod{m}{A}$ by the ideal consisting of morphisms factoring through projective (resp. injective) objects in $\emod{m}{A}$.
\end{notation}

We recall from \cite[Definition~1.3]{J} the notion of Auslander-Reiten triangle in the extension-closed subcategory $\emod{m}{A}$ of $D^b(\mod A)$ as follows. 

\begin{definition}
    A triangle in $D^b(\mod A)$
    $$\XX\to\YY\to\ZZ\xrightarrow{\delta}\XX[1]$$
    with $\XX$, $\YY$, and $\ZZ$ in $\emod{m}{A}$ is called an Auslander-Reiten triangle in $\emod{m}{A}$ starting at $\XX$ and ending at $\ZZ$, if it satisfies the following conditions.
    \begin{enumerate}
    \item[(i)] The morphism $\delta$ is non-zero.
    \item[(ii)] Any non-section morphism $\alpha:\XX\to\XX'$ in $\emod{m}{A}$ factors through the morphism $\XX\to\YY$ in the triangle.
    \item[(iii)] Any non-retraction morphism $\beta:\ZZ'\to\ZZ$ in $\emod{m}{A}$ factors through the morphism $\YY\to\ZZ$ in the triangle.
    \end{enumerate}
\end{definition}

It is pointed out in \cite{INP} that the notion of Auslander-Reiten triangles in an extension-closed subcategory of a triangulated category coincides with that of almost split extensions in extriangulated categories introduced in \cite{INP}, as well as that of Auslander-Reiten $\mathbb{E}$-triangles in extriangulated categories introduced in \cite{ZZ1}. 

\begin{remark}
    Condition (ii) is equivalent to that any non-section morphism $\alpha:\XX\to\XX'$ in $\emod{m}{A}$ satisfies $\alpha\circ\delta[-1]=0$. Condition (iii) is equivalent to that any non-retraction morphism $\beta:\ZZ'\to\ZZ$ in $\emod{m}{A}$ satisfies $\delta\circ\beta=0$.
\end{remark}

For any Auslander-Reiten triangle in $\emod{m}{A}$ as above, both $\XX$ and $\ZZ$ are indecomposable, see \cite[Lemma~2.4]{J}.

\begin{definition}
    We say that $\emod{m}{A}$ has Auslander-Reiten triangles if 
    \begin{enumerate}
        \item[(1)] for any non-projective indecomposable object $\ZZ$ in $\emod{m}{A}$, there is an Auslander-Reiten triangle ending at $\ZZ$, and
        \item[(2)] for any non-injective indecomposable object $\XX$ in $\emod{m}{A}$, there is an Auslander-Reiten triangle starting at $\XX$.
    \end{enumerate}
\end{definition}

For any $\ZZ\in\emod{m}{A}$, we denote by $\pp\ZZ$ (resp. $\ii\ZZ$) its minimal projective (resp. injective) resolution. Then $\pp\ZZ$ (resp. $\ii\ZZ$) is a complex of projective (resp. injective) $A$-modules concentrated in degrees $\leq 0$ (resp. $\geq -(m-1)$).

We refer to \Cref{sec: trun} for the two ways to truncate complexes: stupid truncation $\sigma_{\geq p}$ and $\sigma_{\leq p}$, and canonical truncation $\sigma^{\geq p}$ and $\sigma^{\leq p}$.

\begin{definition}\label{def:mpp}
    The stupid truncation $\ppm(\ZZ):=\sigma_{\geq -m}(\pp\ZZ)$ (resp. $\iim(\ZZ):=\sigma_{\leq 1}(\ii\ZZ)$) is called a minimal projective (resp. injective) presentation of an object $\ZZ$ in $\emod{m}{A}$. 
\end{definition}

By definition, $\ppm(\ZZ)$ is a complex of projective $A$-modules concentrated in degrees $[-m,0]$ with $H^{[-(m-1),0]}(\ppm(\ZZ))\cong\ZZ$, and $\iim(\ZZ)$ is a complex of injective $A$-modules concentrated in degrees $[-(m-1),1]$ with $H^{[-(m-1),0]}(\iim(\ZZ))\cong\ZZ$. In the usual case (i.e., $m=1$), these notions are the usual ones for $A$-modules in $\mod A$.

\begin{definition}\label{def:AR}
    For any $\ZZ\in\emod{m}{A}$, we define
    $$\AR{m}(\ZZ)=\sigma^{\leq 0}(\nu\ppm(\ZZ)[-1])\text{ and }\AR{m}^-(\ZZ)=\sigma^{\geq -(m-1)}(\nu^-\iim(\ZZ)[1]).$$
\end{definition}

In the usual case (i.e., $m=1$), we have $\AR{1}=\tau$ and $\AR{1}^-=\tau^-$, the usual Auslander-Reiten translations in the module category $\mod A$.

\begin{remark}\label{rmk:AR}
    Since $\ppm(\ZZ)$ is concentrated in degrees $[-m,0]$, after applying $\nu[-1]$, the complex $\nu\ppm(\ZZ)[-1]$ is concentrated in degrees $[-(m-1),1]$. Hence, we have
    $$\AR{m}(\ZZ)=H^{[-(m-1),0]}(\nu\ppm(\ZZ)[-1]),$$
    and by canonical truncation of $\nu\ppm(\ZZ)[-1]$, there is a triangle
    \begin{equation}\label{eq:ARdef}
        \AR{m}(\ZZ)\to\nu\ppm(\ZZ)[-1]\to H^0(\nu\ppm(\ZZ))[-1]\to \AR{m}(\ZZ)[1].
    \end{equation}
    Similarly, the complex $\nu^-\iim(\ZZ)[1]$ is concentrated in degrees $[-m,0]$. Hence, we have
    $$\AR{m}^-(\ZZ)=H^{[-(m-1),0]}(\nu^-\iim(\ZZ)[1]),$$
    and there is a triangle
    $$H^{-(m-1)}(\nu^-\iim(\ZZ))[m]\to \nu^-\iim(\ZZ)[1]\to \AR{m}^-(\ZZ)\to H^{-(m-1)}(\nu^-\iim(\ZZ))[m+1].$$
\end{remark}

\begin{remark}\label{rmk:Iyama}
    The notion of $m$-Auslander-Reiten translations $\tau_m$ and $\tau_m^-$ for $\mod A$ were introduced by Iyama \cite{I} (cf. also \cite[Section~3.2]{I2}), where for any $\ZZ\in\mod A$, $\tau_m(\ZZ)=H^{-(m-1)}(\nu\ppm(\ZZ)[-1])$ and $\tau^{-}_m(\ZZ)=H^{0}(\nu^{-}\iim(\ZZ[m-1])[-1])$. It is clear that $\tau_{[m]}$ and $\tau_m$ (resp. $\tau^{-}_{[m]}$ and $\tau^{-}_m$) have different domains. Further, if we regard an object $\ZZ\in\mod A$ as an object in $\emod{m}{A}$ by the canonical inclusion $\mod A\subseteq\emod{m}{A}$, since the complex $\ppm(\ZZ)$ is exact everywhere unless the $0$-th and $(-m)$-th positions, we have $H^i(\nu\ppm(\ZZ)[-1])=0$ for any $-(m-2)\leq i\leq 0$. Hence, by definition, $\AR{m}(\ZZ)=H^{[-(m-1),0]}(\nu\ppm(\ZZ)[-1])=H^{-(m-1)}(\nu\ppm(\ZZ)[-1])[m-1]$. Thus, we have
    $$\tau_m(\ZZ)=(\AR{m}(\ZZ))[-(m-1)].$$
    Similarly, we have
    $$\tau_m^-(\ZZ)=\AR{m}^-(\ZZ[m-1]).$$
\end{remark}

Similar to the usual case (i.e., $m=1$), the correspondence $\ZZ\mapsto\AR{m}(\ZZ)$ does not define an endo-functor of the extended module category $\emod{m}{A}$, but rather a functor between the quotient categories $\uemod{m}{A}$ and $\oemod{m}{A}$ of $\emod{m}{A}$.

\begin{proposition}\label{prop:ARfun}
    The correspondence $\ZZ\mapsto\AR{m}(\ZZ)$ induces an equivalence of additive categories
    $$\AR{m}:\uemod{m}{A}\to\oemod{m}{A},$$
    with quasi-inverse $\AR{m}^-$ induced by the correspondence $\ZZ\mapsto\AR{m}^-(\ZZ)$.
\end{proposition}

\begin{proof}
We consider the subcategory $K^{[-m,0]}(\proj A)$ of $K^b(\proj A)$ consisting of complexes of projective modules concentrated in degrees $[-m,0]$. The cohomology $H^{[-(m-1),0]}$ gives rise to a functor 
$$F\colon K^{[-m,0]}(\proj A)\to\emod{m}{A}.$$
Since $H^{[-(m-1),0]}(\ppm(\ZZ))\cong\ZZ$ for any $\ZZ\in\emod{m}{A}$, this functor is dense. 

For any $\PP_1,\PP_2\in K^{[-m,0]}(\proj A)$, since $\PP_1$ is a complex of projective modules concentrated in degrees $[-m,0]$, we have $\Hom(\PP_1,H^{-m}(\PP_2)[m+1])=0$. Hence, for any morphism $f:H^{[-(m-1),0]}(\PP_1)\to H^{[-(m-1),0]}(\PP_2)$ in $\emod{m}{A}$, there is a morphism $f':\PP_1\to\PP_2$ such that the following diagram of triangles commutes
\[\xymatrix{
H^{-m}(\PP_1)[m]\ar[r]&\PP_1\ar[r]\ar[d]^{f'}& H^{[-(m-1),0]}(\PP_1)\ar[d]^{f}\ar[r]&H^{-m}(\PP_1)[m+1]\\
H^{-m}(\PP_2)[m]\ar[r]&\PP_2\ar[r]& H^{[-(m-1),0]}(\PP_2)\ar[r]&H^{-m}(\PP_2)[m+1]
}\]
where the rows are given by canonical truncation of $\PP_1$ and $\PP_2$, respectively. So $f=H^{[-(m-1),0]}(f')$, which implies that the functor $F$ is full. 

Let $g:\PP_1\to \PP_2$ be a morphism in $K^{[-m,0]}(\proj A)$ satisfying $H^{[-(m-1),0]}(g)=0$ (e.g., take $g=f'$ in the above diagram, then $f=H^{[-(m-1),0]}(g)=0$). So $g$ factors through $H^{-m}(\PP_2)[m]$. By stupid truncation of $\PP_1$, there is a triangle
\[\sigma_{\geq -(m-1)}(\PP_1)\to \PP_1\to \sigma_{\leq -m}(\PP_1)\to \sigma_{\geq -(m-1)}(\PP_1)[1].\]
Applying $\Hom(-,H^{-m}(\PP_2)[m])$ to this triangle, since $\sigma_{\geq -(m-1)}(\PP_1)$ is a complex of projective modules concentrated in degrees $\geq -(m-1)$, we obtain an exact sequence
\[\Hom(\sigma_{\leq -m}(\PP_1),H^{-m}(\PP_2)[m])\to\Hom(\PP_1,H^{-m}(\PP_2)[m])\to 0.\]
Then $g$ factors through $\sigma_{\leq -m}(\PP_1)$. Since $\PP_1\in K^{[-m,0]}(\proj A)$, $\sigma_{\leq -m}(\PP_1)=P[m]$ for some $P\in\proj A$. Conversely, if a morphism $g:\PP_1\to\PP_2$ in $K^{[-m,0]}$ factoring through $P[m]$ for some $P\in\proj A$, since $H^{[-(m-1),0]}(P[m])=0$, one has $H^{[-(m-1),0]}(g)=0$. Therefore, the ideal of $K^{[-m,0]}(\proj A)$ consisting of morphisms factoring through $\add A[m]$ is the kernel of the functor $F$. Hence, $F$ induces an equivalence
$$\overline{F}\colon K^{[-m,0]}(\proj A)/\add A[m]\xrightarrow{\simeq}\emod{m}{A},$$
with the quasi-inverse giving by taking a minimal projective presentation. This equivalence restricts to an equivalence from $\proj A$ to $\proj A$. So $\overline{F}$ induces an equivalence
$$\overline{\overline{F}}\colon K^{[-m,0]}(\proj A)/\add (A\oplus A[m])\xrightarrow{\simeq}\uemod{m}{A}.$$
Dually, there is an equivalence also induced by the cohomology $H^{[-(m-1),0]}$:
$$\overline{\overline{E}}\colon K^{[-(m-1),1]}(\inj A)/\add (\nu A[-1]\oplus \nu A[m-1])\xrightarrow{\simeq}\oemod{m}{A}$$
with the quasi-inverse given by taking a minimal injective presentation. Therefore, the composition of equivalences
$$\overline{\overline{E}}\circ[-1]\circ\nu\circ\overline{\overline{F}}^{-}:\uemod{m}{A}\xrightarrow{\simeq}\oemod{m}{A}$$
sends $\ZZ$ to $\AR{m}(\ZZ)$, with the quasi-inverse $\overline{\overline{F}}\circ\nu^{-}\circ[1]\circ\overline{\overline{E}}^{-}$ sending $\ZZ$ to $\AR{m}^-(\ZZ)$.
\end{proof}

We proceed to summarise some properties of $\AR{m}$ and $\AR{m}^-$.

\begin{proposition}\label{ARprop}
    Let $\ZZ$ be an indecomposable object in $\emod{m}{A}$.
    \begin{enumerate}
        \item[(i)] $\AR{m}(\ZZ)$ has no nonzero injective objects in $\emod{m}{A}$ as direct summands.
        \item[(ii)] $\AR{m}^-(\ZZ)$ has no nonzero projective objects in $\emod{m}{A}$ as direct summands.
        \item[(iii)] $\ZZ$ is projective in $\emod{m}{A}$ if and only if $\AR{m}(\ZZ)=0$.
        \item[(iv)] $\ZZ$ is injective in $\emod{m}{A}$ if and only if $\AR{m}^-(\ZZ)=0$.
        \item[(v)] If $\ZZ$ is not projective in $\emod{m}{A}$, then $\AR{m}(\ZZ)$ is indecomposable and $$\AR{m}^{-}\AR{m}(\ZZ)\cong\ZZ.$$
        \item[(vi)] If $\ZZ$ is not injective in $\emod{m}{A}$, then $\AR{m}^{-}(\ZZ)$ is indecomposable and $$\AR{m}\AR{m}^{-}(\ZZ)\cong\ZZ.$$
    \end{enumerate}
\end{proposition}

\begin{proof}
    For (i), suppose, to the contrary, that there exists a nonzero injective object $\ZZ_1$ in $\emod{m}{A}$ such that $\AR{m}(\ZZ)\cong\ZZ_1\oplus\ZZ_2$ for some $\ZZ_2\in\emod{m}{A}$. By \Cref{lem:enough}, $\ZZ_1= \nu P[m-1]$ for some $P\in\proj A$. Rotating the triangle~\eqref{eq:ARdef}, we obtain a triangle
    \[H^0(\nu\ppm(\ZZ))[-2]\to(\nu P[m-1])\oplus\ZZ_2\to\nu\ppm(\ZZ)[-1]\to H^0(\nu\ppm(\ZZ))[-1].\]
    Since $\Hom(H^0(\nu\ppm(\ZZ))[-2],\nu P[m-1])\cong\Hom(H^0(\nu\ppm(\ZZ)),\nu P[m+1])=0$, $\nu P[m-1]$ is also a direct summand of $\nu\ppm(\ZZ)[-1]$. So $P[m]$ is a direct summand of $\ppm(\ZZ)$, which contradicts the minimality of the projective presentation $\ppm(\ZZ)$ of $\ZZ\in\emod{m}{A}$. The assertion (ii) can be proved similarly.

    For (iii), if $\ZZ$ is projective in $\emod{m}{A}$, by \Cref{lem:enough}, $\ZZ$ is isomorphic to some $P\in\proj A$. So $\ppm(\ZZ)=\pp\ZZ\cong P$. Then $$\AR{m}(\ZZ)=\sigma^{\leq 0}(\nu\ppm(\ZZ)[-1])=\sigma^{\leq 0}(\nu P[-1])=0.$$
    Conversely, if $\AR{m}(\ZZ)=0$, then by \Cref{prop:ARfun}, $\ZZ$ is zero in $\uemod{m}{A}$. So $\ZZ$ is projective in $\emod{m}{A}$. The assertion (iv) can be proved similarly.

    For (v), since $\ZZ$ is indecomposable and not projective in $\emod{m}{A}$, by \Cref{prop:ARfun}, $\AR{m}(\ZZ)$ is indecomposable in $\oemod{m}{A}$. By (i), any nonzero direct summand of $\AR{m}(\ZZ)$ is not injective in $\emod{m}{A}$. Hence, $\AR{m}(\ZZ)$ is indecomposable in $\emod{m}{A}$. Also by \Cref{prop:ARfun}, we have $\AR{m}^{-}\AR{m}(\ZZ)\cong\ZZ$ in $\uemod{m}{A}$. By (ii), $\AR{m}^{-}\AR{m}(\ZZ)$ has no nonzero direct summands which are projective in $\emod{m}{A}$. So we have $\AR{m}^{-}\AR{m}(\ZZ)\cong\ZZ$ in $\emod{m}{A}$. The assertion (vi) can be proved similarly.
\end{proof}

Now we prove the main theorem of this section.

\begin{theorem}\label{thm:AR}
    Let $\ZZ$ be an indecomposable object in $\emod{m}{A}$. If $\ZZ$ is not projective in $\emod{m}{A}$, there is an Auslander-Reiten triangle in $\emod{m}{A}$
    \[\AR{m}(\ZZ)\to\YY\to\ZZ\to\AR{m}(\ZZ)[1].\]
    If $\ZZ$ is not injective in $\emod{m}{A}$, there is an Auslander-Reiten triangle in $\emod{m}{A}$
    \[\ZZ\to\WW\to\AR{m}^{-}(\ZZ)\to\ZZ[1].\]
    Consequently, the category $\emod{m}{A}$ has Auslander-Reiten triangles. 
\end{theorem}

\begin{proof}
    We only show the first assertion, since the second one can be proved similarly. We adopt the method in \cite{J} to $\emod{m}{A}$.

    Let $K(\proj A)$ (resp. $K(\inj A)$) be the unbounded homotopy category of complexes of modules in $\proj A$ (resp. $\inj A$). The Nakayama functor $\nu$ induces an equivalence
    $$\nu:K(\proj A)\to K(\inj A),$$
    whose restriction to the bounded categories is the equivalence~\eqref{Naks}.

    We regard $\ZZ$ as a compact object in $K(\inj A)$ by the isomorphism $\ZZ\cong\ii\ZZ$ in the unbounded derived category $D(\mod A)$ of $\mod A$, see \cite[Lemma 2.1]{KL}. Then by \cite[Proposition 6.2]{KL}, there is an Auslander-Reiten triangle in $K(\inj A)$
    \begin{equation}\label{eq:KL}
        (\nu\pp\ZZ)[-1]\to \YY'\to\ZZ\xrightarrow{\delta'} \nu\pp\ZZ.
    \end{equation}
    By stupid truncation of $\pp\ZZ$ and applying $\nu$, we get a triangle
    \begin{equation}\label{eq:tri-2}
        \nu\ppm(\ZZ)\xrightarrow{f}\nu\pp\ZZ\to \nu\sigma_{\leq -(m+1)}(\pp\ZZ)\to \nu\ppm(\ZZ)[1].
    \end{equation}
    For any $\ZZ'\in \emod{m}{A}$, applying $\Hom(\ZZ',-)$ to this triangle, we get a long exact sequence
    \[\begin{array}{cccc}
         & \Hom(\ZZ',\nu\sigma_{\leq -(m+1)}(\pp\ZZ)[-1])&\to&\Hom(\ZZ',\nu\ppm(\ZZ)) \\
        \xrightarrow{\Hom(\ZZ',f)} & \Hom(\ZZ',\nu\pp\ZZ)&\to& \Hom(\ZZ',\nu\sigma_{\leq -(m+1)}(\pp\ZZ)).
    \end{array}\]
    Since $\nu\sigma_{\leq -(m+1)}(\pp\ZZ)$ is an complex of injective $A$-modules concentrated in degrees $\leq -(m+1)$ and $\ZZ'\in \emod{m}{A}=\D^{[-(m-1),0]}$, we have
    $$\Hom(\ZZ',\nu\sigma_{\leq -(m+1)}(\pp\ZZ)[-1])=0=\Hom(\ZZ',\nu\sigma_{\leq -(m+1)}(\pp\ZZ)).$$
    Hence, $\Hom(\ZZ',f)$ is a functorial isomorphism.
    It follows that the morphism $\delta':\ZZ\to \nu\pp\ZZ$ in the triangle~\eqref{eq:KL} factors through $f$, i.e., there is a morphism $\delta'':\ZZ\to \nu\ppm(\ZZ)$ such that $\delta'=f\circ\delta''$. See the right triangle in the commutative diagram~\eqref{eq:cd}.
    
    We claim that $\delta''$ satisfies the following property.
    \begin{enumerate}
        \item[($\star$)] $\delta''\circ \gamma=0$ for any non-retraction morphism $\gamma:\ZZ'\to\ZZ$ in $D^b(\mod A)$. 
    \end{enumerate}
    Indeed, since $\gamma$ is not a retraction and $\ZZ'\cong\ii\ZZ'$ can be regarded as an object in $K(\inj A)$, by the Auslander-Reiten triangle~\eqref{eq:KL}, we have $\delta'\circ\gamma=0$. Then $\delta''\circ\gamma=0$, since $\Hom(\ZZ',f)$ is a functorial isomorphism.
    
    Extend $\delta'':\ZZ\to \nu\ppm(\ZZ)$ to a triangle
    \[\nu\ppm(\ZZ)[-1]\to\YY''\to\ZZ\xrightarrow{\delta''}\nu\ppm(\ZZ).\]
    Since $\ZZ$ is not projective in $\emod{m}{A}$, by definition, there is a non-zero morphism $\alpha:\ZZ\to\XX'''[1]$ in $D^b(\mod A)$ for some $\XX'''\in\emod{m}{A}$. Extending it to a triangle
    \[\XX'''\to\YY'''\xrightarrow{g}\ZZ\xrightarrow{\alpha}\XX'''[1]\]
    in $D^b(\mod A)$, we have that $g$ is not a retraction. So by ($\star$), we have $\delta''\circ g=0$. Hence, $\delta''$ factors through $\alpha$, i.e., there is a morphism $h:\XX'''[1]\to \nu\ppm(\ZZ)$ such that $h\circ \alpha=\delta''$. See the middle triangle in the commutative diagram~\eqref{eq:cd}.
    
    Shifting the triangle~\eqref{eq:ARdef}, we get the following triangle
    \begin{equation}\label{eq:01}
        \AR{m}(\ZZ)[1]\xrightarrow{\beta}\nu\ppm(\ZZ)\to H^0(\nu\ppm(\ZZ))\to \AR{m}(\ZZ)[2].
    \end{equation}
    Since $\Hom(\XX'''[1],H^0(\nu\ppm(\ZZ)))=0$ due to $\XX'''[1]\in\D^{[-m,-1]}$ and $H^0(\nu\ppm(\ZZ))\in\D^{[0,0]}$, the morphism $h$ factors through $\beta$, i.e., there exists a morphism 
    $\gamma:\XX'''[1]\to \AR{m}(\ZZ)[1]$ such that $\beta\circ\gamma=h$. See the left triangle in the commutative diagram~\eqref{eq:cd}.
    
    \begin{equation}\label{eq:cd}
        \xymatrix{
        &\XX'''[1]\ar@{-->}[dl]_{\gamma}\ar@{-->}[rd]_-{h}&\ZZ\ar[l]_-{\alpha}\ar[r]^-{\delta'}\ar@{-->}[d]_-{\delta''}&\nu\pp\ZZ\\
    \AR{m}(\ZZ)[1]\ar[rr]_{\beta}&&\nu\ppm(\ZZ)\ar[ru]_{f}
    }
    \end{equation}
    
    Take $\delta=\gamma\circ\alpha$ and extend it to a triangle
    \[\AR{m}(\ZZ)\to\YY\to \ZZ\xrightarrow{\delta} \AR{m}(\ZZ)[1].\]
    We claim that this is an Auslander-Reiten triangle in $\emod{m}{A}$. Indeed, since by \Cref{ARprop}~(v), $\AR{m}(\ZZ)$ is indecomposable, by \cite[Theorem~2.9]{INP}, we only need to show that for any non-retraction morphism $z:\ZZ'\to\ZZ$ in $\emod{m}{A}$, one has $\delta\circ z=0$. By ($\star$), we have $\delta''\circ z=0$. Since $\beta\circ\delta=\beta\circ\gamma\circ \alpha=h\circ\alpha=\delta''$, we have $\beta\circ(\delta\circ z)=\delta''\circ z=0$. So by triangle~\eqref{eq:01}, the morphism $\delta\circ z:\ZZ'\to \AR{m}(\ZZ)[1]$ factors through $H^0(\nu\ppm(\ZZ))[-1]$. However, since $\ZZ'\in \emod{m}{A}=\D^{[-(m-1),0]}$ and $H^0(\nu\ppm(\ZZ))[-1]\in \D^{[1,1]}$, we have $\Hom(\ZZ',H^0(\nu\ppm(\ZZ))[-1])=0$. Therefore, $\delta\circ z=0$ as required.
\end{proof}

By \Cref{thm:AR} and \cite[Theorem~3.6]{INP}, we have the following Auslander-Reiten formula for $\emod{m}{A}$. 

\begin{corollary}\label{cor:AR}
    For any objects $\XX$ and $\YY$ in $\emod{m}{A}$, there are functorial isomorphisms
    $$\Hom_{\uemod{m}{A}}(\XX,\YY)\cong D\mathbb{E}(\YY,\AR{m}(\XX)),$$
    and
    $$\Hom_{\oemod{m}{A}}(\XX,\YY)\cong D\mathbb{E}(\AR{m}^-(\YY),\XX).$$
\end{corollary}

\begin{remark}
    For any objects $\XX$ and $\YY$ in $\emod{m}{A}$, there are functorial isomorphisms
    $$\mathbb{E}(\YY,\XX)\cong D\Hom_{\uemod{m}{A}}(\AR{m}^-(\XX),\YY)\cong D\Hom_{\oemod{m}{A}}(\XX,\AR{m}(\YY)).$$
    This is because, by \Cref{ARprop}, $\YY\cong\AR{m}^-(\AR{m}(\YY))\oplus P$ for some $P\in\proj A$. Then by \Cref{cor:AR}, there is a functorial isomorphism
    $$\Hom_{\oemod{m}{A}}(\XX,\AR{m}(\YY))\cong D\mathbb{E}(\AR{m}^-(\AR{m}(\YY)),\XX)\cong D\mathbb{E}(\YY,\XX),$$
    which shows $\mathbb{E}(\YY,\XX)\cong  D\Hom_{\oemod{m}{A}}(\XX,\AR{m}(\YY))$. Similarly, we can also show $\mathbb{E}(\YY,\XX)\cong D\Hom_{\uemod{m}{A}}(\AR{m}^-(\XX),\YY)$.
\end{remark}

We refer to \cite[Section~3.3]{INP} for the notion of Auslander-Reiten quiver of a Krull–Schmidt extriangulated category that has Auslander-Reiten triangles.

\begin{example}
    When the algebra $A$ is hereditary, the Auslander-Reiten quiver of $\emod{m}{A}$ is the full subquiver of the Auslander-Reiten quiver of $D^b(\mod A)$ consisting of the vertices indexed by indecomposable objects $M[i]$ for all $M\in\mod A$ and $0\leq i\leq m-1$. This is because, due to that $A$ is hereditary, the indecomposable objects in $\emod{m}{A}$ are the $M[i]$'s. Moreover, for any indecomposable non-projective object $\ZZ=M[i]$, we have $(\nu\pp\ZZ)[-1]\in\emod{m}{A}$. Hence, the Auslander-Reiten triangle \eqref{eq:KL} is also an Auslander-Reiten triangle in $\emod{m}{A}$. Dually, any Auslander-Reiten triangle in $D^b(\mod A)$ starting at an indecomposable non-injective object in $\emod{m}{A}$ is also an Auslander-Reiten triangle in $\emod{m}{A}$.
    
    For instance, if $A=\k Q$ for $Q=1\to2\to 3$, then the Auslander-Reiten quiver of $\emod{2}{A}$ is as follows.
    
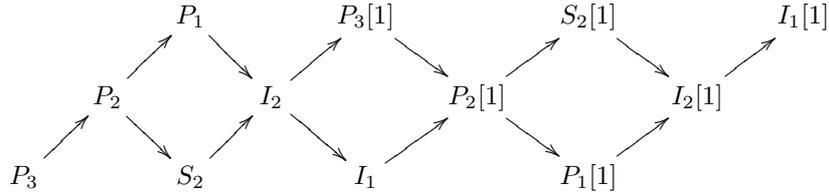
\[\xymatrix@R=.5cm@C=.5cm{
    &&P_1\ar[rd]&&P_3[1]\ar[rd]&&S_2[1]\ar[rd]&&I_1[1]\\
    &P_2\ar[ru]\ar[rd]&&I_2\ar[ru]\ar[rd]&&P_2[1]\ar[ru]\ar[rd]&&I_2[1]\ar[ru]\\
    P_3\ar[ru]&&S_2\ar[ru]&&I_1\ar[ru]&&P_1[1]\ar[ru]
    }\]
\end{example}

\begin{example}\label{ex:AR}
    Let $A=\k Q/I$, where $Q$ is the quiver $\xymatrix{1 \ar@/^/[r]^{\alpha} & 2 \ar@/^/[l]^{\beta}}$ and $I=\langle\alpha\beta\rangle$. To the vertices $1$ and $2$, the corresponding indecomposable projective $A$-modules are $P_1=\begin{smallmatrix}
        1\\2
    \end{smallmatrix}$ and $P_2=\begin{smallmatrix}
        2\\1\\2
    \end{smallmatrix}$ respectively, while the corresponding indecomposable injective $A$-modules are $I_1=\begin{smallmatrix}
        2\\1
    \end{smallmatrix}$ and $I_2=\begin{smallmatrix}
        2\\1\\2
    \end{smallmatrix}=P_2$ respectively. Then the Auslander-Reiten quiver of the 2-extended module category $\emod{2}{A}$ is as follows, where two $(0\to I_1)$'s, $(S_2\to 0)$'s and $(0\to S_2)$'s are identified, respectively.
    
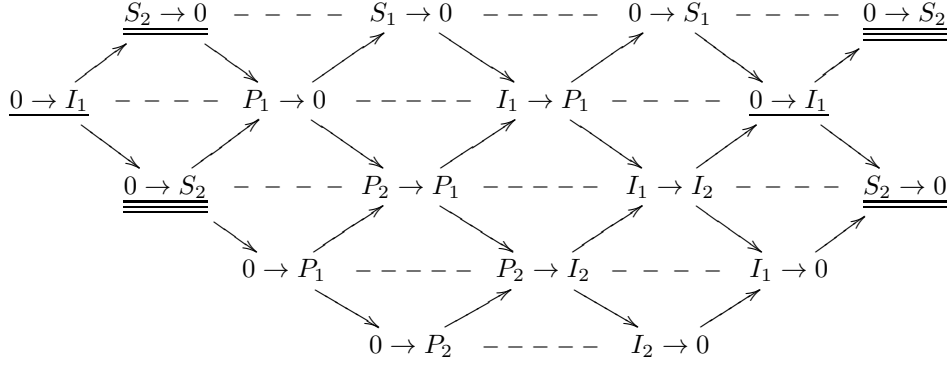
\[\xymatrix@C=.25cm@R=.5cm{
    &\underline{\underline{S_2\to 0}}\ar[rd]\ar@{}[rr]^(.25){}="a"^(.75){}="b" \ar@{--} "a";"b"
    &&S_1\to 0\ar[rd]\ar@{}[rr]^(.25){}="a"^(.75){}="b" \ar@{--} "a";"b"
    &&0\to S_1\ar[rd]\ar@{}[rr]^(.25){}="a"^(.75){}="b" \ar@{--} "a";"b"
    &&\underline{\underline{\underline{0\to S_2}}}\\
    \underline{0\to I_1}\ar[ru]\ar[rd]\ar@{}[rr]^(.25){}="a"^(.75){}="b" \ar@{--} "a";"b"
    &&P_1\to 0\ar[ru]\ar[rd]\ar@{}[rr]^(.25){}="a"^(.75){}="b" \ar@{--} "a";"b"
    &&I_1\to P_1\ar[ru]\ar[rd]\ar@{}[rr]^(.25){}="a"^(.75){}="b" \ar@{--} "a";"b"
    &&\underline{0\to I_1}\ar[ru]\ar[rd]\\
    &\underline{\underline{\underline{0\to S_2}}}\ar[ru]\ar[rd]\ar@{}[rr]^(.25){}="a"^(.75){}="b" \ar@{--} "a";"b"
    &&P_2\to P_1\ar[ru]\ar[rd]\ar@{}[rr]^(.25){}="a"^(.75){}="b" \ar@{--} "a";"b"
    &&I_1\to I_2\ar[ru]\ar[rd]\ar@{}[rr]^(.25){}="a"^(.75){}="b" \ar@{--} "a";"b"
    &&\underline{\underline{S_2\to 0}}\\
    &&0\to P_1\ar[ru]\ar[rd]\ar@{}[rr]^(.25){}="a"^(.75){}="b" \ar@{--} "a";"b"
    &&P_2\to I_2\ar[ru]\ar[rd]\ar@{}[rr]^(.25){}="a"^(.75){}="b" \ar@{--} "a";"b"
    &&I_1\to 0\ar[ru]\\
    &&&0\to P_2\ar[ru]\ar@{}[rr]^(.25){}="a"^(.75){}="b" \ar@{--} "a";"b"
    &&I_2\to 0\ar[ru]
    }\]
    Here, each indecomposable object is denoted by a 2-term complex $X_1\to X_2$ of modules $X_1, X_2\in\mod A$, with the morphism nonzero in the radical, unless $X_1$ or $X_2$ is zero. Since $A$ is a gentle algebra, according to the classification of indecomposable objects in $D^b(\mod A)$ provided in \cite{BM}, the indecomposable objects in $\emod{2}{A}$ are the 14 objects depicted in the quiver above.
    
    Let $\PP$ be the 3-term complex $\left(P_2\xrightarrow{\beta\alpha} P_2\xrightarrow{\alpha} P_1\right)\oplus\left(P_2\to 0\to 0\right)$ in $K^b(\proj A)$. It is straightforward to check that $\PP$ is silting. Moreover, we have 
    $$\T(\PP)=\add\{I_1\to P_1,0\to S_1,S_1\to 0\},$$
    and
    $$\F(\PP)=\add\{0\to S_2,0\to P_1,0\to P_2,P_1\to 0,P_2\to P_1,P_2\to I_2,I_2\to 0\}.$$
    Since $H^{[-1,0]}(\PP)=(I_1\to P_1)$, by \Cref{prop:genTP}, $\T(\PP)=\Fac_2(I_1\to P_1)$ and has enough projectives $\add(I_1\to P_1)$. The canonical triangle of $I_1\to 0$ with respect to $(\T(\PP),\F(\PP))$ is
    $$(I_1\to P_1)\to(I_1\to 0)\to (P_1\to 0)\to (I_1\to P_1)[1].$$
    So by \Cref{tf}, we have $\mathfrak{t}(I_1\to 0)=(I_1\to P_1)$. On the other hand, since $(I_2\to 0)\in\F(\PP)$, we have $\mathfrak{t}(I_2\to 0)=0$. Hence, by \Cref{prop:tpp1}~(d), $\T(\PP)$ has enough injectives $\add(I_1\to P_1)$. Similarly, we have $H^{[-1,0]}(\nu\PP[-1])=(P_1\oplus I_2\to 0)$ and $\F(\PP)=\Sub_2(P_1\oplus I_2\to 0)$, which has enough injectives $\add(P_1\oplus I_2\to 0)$ and enough projectives $\add(0\to P_1\oplus P_2)$.
\end{example}

We conclude this section with a generalization of \cite[Proposition~5.8]{AS}.

\begin{proposition}\label{prop:as1}
    For any objects $\XX$ and $\YY$ in $\emod{m}{A}$, the following are equivalent.
    \begin{itemize}
        \item[(1)] $\Hom(\XX,\AR{m}(\YY)[j])=0$ for any $j\leq 0$.
        \item[(2)] $\Hom(\Fac_{m}(\XX),\AR{m}(\YY)[j])=0$ for any $j\leq 0$.
        \item[(3)] $\Hom(\Fac_{m}(\XX),\AR{m}(\YY))=0$.
        \item[(4)] $\mathbb{E}(\YY,\Fac_{m}(\XX))=0$. 
    \end{itemize} 
\end{proposition}

\begin{proof}
    (1)$\implies$(2): This is a direct application of \Cref{lem:asfac}.

    (2)$\implies$(3): Trivial. 
    
    (3)$\implies$(4): This follows directly from \Cref{cor:AR}.

    (4)$\implies$(1): Since both $\XX$ and $\AR{m}(\YY)$ are objects in $\emod{m}{A}$, we have $\Hom(\XX,\AR{m}(\YY)[j])=0$ for any $j\leq -m$. Now, let $-(m-1)\leq j\leq 0$ and take an arbitrary $f\in\Hom(\XX[-j],\AR{m}(\YY))$. Extend $f$ to a triangle in $D^b(\mod A)$
    $$\XX[-j]\xrightarrow{f} \AR{m}(\YY)\to \ZZ\xrightarrow{g} \XX[-j+1].$$
    Then $\ZZ\in\add\AR{m}(\YY)*\add\XX[-j+1]\subseteq\D^{[-(m-1),0]}*\D^{[-m+j,j-1]}\subseteq\D^{[-m+j,0]}$. So by canonical truncation of $\ZZ$, there is a triangle
    $$\ZZ'\xrightarrow{h}\ZZ\to \ZZ'' \to \ZZ'[1].$$
    with $\ZZ'=H^{[-m+j,-m]}(\ZZ)$ and $\ZZ''=H^{[-(m-1),0]}(\ZZ)$. Taking the composition $g\circ h$ and by the octahedral axiom, we have the following commutative diagram of triangles
    \begin{equation}\label{eq:cd02}
    \xymatrix{
    & \ZZ''[-1]\ar@{=}[r]\ar[d] & \ZZ''[-1]\ar[d]\\
    \XX[-j]\ar@{=}[d]\ar[r] & \YY'\ar[r]\ar[d]^{f'}\ar[r] & \ZZ'\ar[r]^{g\circ h\hspace{.6cm}}\ar[d]^{h} & \XX[-j+1]\ar@{=}[d]\\
    \XX[-j]\ar[r]^{f} & \AR{m}(\YY)\ar[r]\ar[d] & \ZZ\ar[r]^{g\hspace{.8cm}}\ar[d] & \XX[-j+1]\\
    & \ZZ''\ar@{=}[r] & \ZZ''
    }
    \end{equation}
    We claim that $\YY'\in\Fac_m(\XX)$. Indeed, by the triangle in the second column of diagram~\eqref{eq:cd02}, we have $$\YY'\in\add \ZZ''[-1]*\add \AR{m}(\YY)\subseteq\D^{[-(m-2),1]}*\D^{[-(m-1),0]} \subseteq \D^{[-(m-1),1]}.$$
    On the other hand, by the triangle in the second row of diagram~\eqref{eq:cd02}, we have $$\YY'\in\add\XX[-j]*\add \ZZ'\subseteq\D^{[-(m-1)+j,j]}*\D^{[-m+j,-m]} \subseteq\D^{[-m+j,j]}.$$
    So we have $\YY'\in \D^{[-(m-1),1]}\cap\D^{[-m+j,j]}=\D^{[-(m-1),j]}\subseteq\emod{m}{A}$. Then $\YY'[j]\in\D^{[-(m-1)-j,0]}\subseteq\emod{m}{A}$. Hence, by \Cref{exm:fac}~(2), we only need to show $\YY'[j]\in\Fac_{m+j}(\XX)$. Let $\WW=\ZZ'[-m]\in \D^{[j,0]}\subseteq\emod{m}{A}$. Then $\WW[m-1+j]=\ZZ'[-1+j]\in\D^{[-(m-1),-(m-1)-j]}\subseteq\emod{m}{A}$. Shifting and rotating the triangle in the second row of diagram~\eqref{eq:cd02}, we have a triangle
    \begin{equation}\label{eq:tri01}
        \WW[m-1+j]\to\XX\to \YY'[j]\to \WW[m+j].
    \end{equation}
    So by definition, $\YY'[j]\in\Fac_1(\XX)$. If $j=-(m-1)$, we are done. If $j>-(m-1)$, by \Cref{exm:fac}~(1), $\WW[m-1+j]\in\Fac_{m-1+j}(\XX)$. In this case, by triangle~\eqref{eq:tri01} and \Cref{eq:rec}, we have $\YY'[j]\in\Fac_{m+j}(\XX)$. Therefore, in any case, we have $\YY'\in \Fac_{m}(\XX)$. 
    
    Thus, $\mathbb{E}(\YY,\YY')=0$. So by \Cref{cor:AR}, $\Hom_{\oemod{m}{A}}(\YY',\AR{m}(\YY))=0$. Therefore, the morphism $f'$ in the diagram~\eqref{eq:cd02} factors through an injective object $\VV$ in $\emod{m}{A}$, i.e., there are morphisms $g':\YY'\to \VV$ and $h':\VV\to \AR{m}(\YY)$ such that $f'=h'\circ g'$. 
    $$\xymatrix{
    &\YY'\ar[dd]^{f'}\ar@{-->}[ld]_{g'}\\
    \VV\ar@{-->}[rd]_{h'}\\
    &\AR{m}(\YY)
    }$$
    Applying $\Hom(-,\VV)$ to the triangle in the second column of diagram~\eqref{eq:cd02}, we obtain an exact sequence
    $$\Hom(\AR{m}(\YY),\VV)\xrightarrow{\Hom(f',\VV)}\Hom(\YY',\VV)\to\Hom(\ZZ''[-1],\VV),$$
    where the last item is zero because $\ZZ''\in\emod{m}{A}$ and $\VV$ is injective in $\emod{m}{A}$. Thus, the map $\Hom(f',\VV)$ is surjective. So the morphism $g'$ factors through $f'$, i.e., there is $h'':\AR{m}(\YY)\to \VV$ such that $g'=h''\circ f'$. Therefore, we have $f'=h'\circ h''\circ f'$, which implies $f'=(h'\circ h'')^p\circ f'$ for any positive integer $p$. However, since by \Cref{ARprop}~(i), there is no nonzero direct summand of $\AR{m}(\YY)$ which is injective in $\emod{m}{A}$, the morphism $h'\circ h'':\AR{m}(\YY)\to\AR{m}(\YY)$ which factors through $\VV$ is in the radical. Hence, $h'\circ h''$ is nilpotent (since the endomorphism algebra $\End(\AR{m}(\YY))$ is finite-dimensional). Thus, we have $f'=0$, which implies $f=0$. Therefore, we have $\Hom(\XX,\AR{m}(\YY)[j])=0$ for any $j\leq 0$.
\end{proof}

\section{\texorpdfstring{$\tau$}\ -tilting theory for extended module categories}\label{sec:tau-tilting}

In this section, we generalize the $\tau$-tilting theory from $\mod A$ to $\emod{m}{A}$. Note that the extriangulated category $\emod{m}{A}$ admits arbitrary negative extensions $\mathbb{E}^j(\XX,\YY)=\Hom(\XX,\YY[j])$ for $j<0$. We refer to \cite{GNP} for more discussions on negative extensions on extriangulated categories.

\begin{definition}
    An object $\XX$ in $\emod{m}{A}$ is called $\AR{m}$-rigid if $$\Hom(\XX,\AR{m}(\XX))=0.$$
    A $\AR{m}$-rigid object $\XX$ in $\emod{m}{A}$ is called positive $\AR{m}$-rigid if in addition
    \begin{equation}\label{eq:ptr}
        \mathbb{E}^j(\XX,\AR{m}(\XX))=0,\ j<0.
    \end{equation}
\end{definition}

In the above definition, both $\XX$ and $\AR{m}(\XX)$ are in $\emod{m}{A}$, so for any $j\leq -m$, we have $\Hom(\XX,\AR{m}(\XX)[j])=0$. Therefore, the range of values for $j$ in condition~\eqref{eq:ptr} can be replaced with $-(m-1)\leq j\leq -1.$

\begin{example}
    In the usual case (i.e., $m=1$), any $\AR{m}$-rigid module is positive $\AR{m}$-rigid. However, this is not true in general. In \Cref{ex:AR}, since $\AR{2}(0\to S_1)=(S_1\to 0)$, the object $(0\to S_1)$ in $\emod{2}{A}$ is $\AR{2}$-rigid but not positive $\AR{2}$-rigid.
\end{example}

\begin{remark}
    By \Cref{prop:as1}, an object $\XX\in\emod{m}{A}$ is positive $\AR{m}$-rigid if and only if $\mathbb{E}(\XX,\Fac_m(\XX))=0$. This indicates that in defining positive $\AR{m}$-rigid objects, it is possible to avoid assuming both the existence of negative extensions and the existence of the Auslander-Reiten translation.
\end{remark}

The following notion is a generalization of $\tau$-rigid pairs in the module category.

\begin{definition}\label{def:taup}
    A pair $(\XX,P)$ of $\XX\in\emod{m}{A}$ and $P\in\proj A$ is called a positive $\AR{m}$-rigid pair in $\emod{m}{A}$ if $\XX$ is positive $\AR{m}$-rigid and
    \begin{equation}\label{eq:ptrp}
        \Hom(P,\XX[i])=0,\ i\leq 0.
    \end{equation}
\end{definition}

Since $P$ is a complex of projective modules concentrated in degree 0 and $\XX\in\emod{m}{A}$, the range of values for $i$ in condition~\eqref{eq:ptrp} can be replaced with $i\in\mathbb{Z}$, or with $-(m-1)\leq i\leq 0$.

For any $\XX\in\emod{m}{A}$, we define two subcategories of $\emod{m}{A}$ as follows:
$$\XX^{\bot_{\leq 0}}=\{\YY\in\emod{m}{A}\mid \Hom(\XX,\YY[i])=0,\ \forall i\leq 0\},$$
and
$$^{\bot_{\leq 0}}\XX=\{\YY\in\emod{m}{A}\mid \Hom(\YY,\XX[i])=0,\ \forall\ i\leq 0\}.$$
By \Cref{prop:as1}, for any positive $\AR{m}$-rigid object $\XX$ in $\emod{m}{A}$, we have 
\begin{equation}\label{eq:inc}
    \Fac_m(\XX)\subseteq{}^{\bot_{\leq 0}}(\AR{m}(\XX)).
\end{equation}
Then for any positive $\AR{m}$-rigid pair $(\XX,P)$, we have
\begin{equation}
    \Fac_m(\XX)\subseteq{}^{\bot_{\leq 0}}(\AR{m}(\XX))\cap P^{\bot_{\leq 0}}.
\end{equation}
Indeed, for any $\ZZ\in\Fac_m(\XX)$, by definition, there are triangles
$$\ZZ_i\to\XX_i\to\ZZ_{i-1}\to\ZZ_{i}[1],\ 1\leq i\leq m,$$
with $\ZZ_0=\ZZ,\ZZ_1,\cdots,\ZZ_m\in\emod{m}{A}$ and $\XX_1,\cdots,\XX_m\in\add\XX$. By applying $\Hom(P,-)$ to these triangles, we obtain isomorphisms
$$\Hom(P,\ZZ_{i-1}[j])\cong \Hom(P,\ZZ_{i}[j+1]),\ 1\leq i\leq m,\ j\in\mathbb{Z}.$$
Then for any $-(m-1)\leq j\leq 0$, we have
$$\Hom(P,\ZZ[j])\cong \Hom(P,\ZZ_m[m+j])=0,$$
due to $\ZZ_m[m+j]\in \D^{\leq -1}$. This implies $\ZZ\in P^{\bot_{\leq 0}}$.

\begin{definition}\label{def:tt}
    A positive $\AR{m}$-rigid object $\XX$ in $\emod{m}{A}$ is called $\AR{m}$-tilting if 
    \begin{equation}\label{eq:deftt}
    {}^{\bot_{\leq 0}}(\AR{m}(\XX))\subseteq\Fac_{m}(\XX).
    \end{equation}
    A positive $\AR{m}$-rigid pair $(\XX,P)$ is called $\AR{m}$-tilting if 
    \begin{equation}\label{eq:defttp}
    {}^{\bot_{\leq 0}}(\AR{m}(\XX))\cap P^{\bot_{\leq 0}}\subseteq\Fac_{m}(\XX).
    \end{equation}
\end{definition}

In the usual case (i.e., $m=1$), $\AR{m}$-tilting objects in $\emod{m}{A}$ are exactly $\tau$-tilting modules in $\mod A$, and $\AR{m}$-tilting pairs defined here coincide with usual ones, see \cite[Theorem~2.12 and Corollary~2.13]{AIR}.

\begin{example}\label{ex:taut}
    In \Cref{ex:AR}, for the positive $\AR{2}$-rigid pair $(I_1\to P_1,0)$, there are exactly two basic $\AR{2}$-tilting pairs in $\emod{2}{A}$ that contains it as a direct summand: $\left((P_2\to P_1)\oplus(I_1\to P_1),0\right)$ and $(I_1\to P_1,P_2)$. For the positive $\AR{2}$-rigid pair $(0,P_2)$, there are exactly three basic $\AR{2}$-tilting pairs in $\emod{2}{A}$ that contains it as a direct summand: $(I_1\to P_1,P_2)$, $(S_1\to 0,P_2)$ and $(0,P_1\oplus P_2)$.
\end{example}

We denote by
\begin{itemize}
    \item $\tp$ the set of (isoclasses of) basic $\AR{m}$-tilting pairs in $\emod{m}{A}$,
    \item $\ft$ the set of functorially finite $s$-torsion pairs in $\emod{m}{A}$,
    \item $\sil$ the set of (isoclasses of) basic $(m+1)$-term silting complexes in $K^b(\proj A)$.
\end{itemize}
Recall from \Cref{def:mpp} that $\ppm(\XX)$ denotes a minimal projective presentation of an object $\XX$ in $\emod{m}{A}$.

\begin{theorem}\label{thm:bi}
    There is the following commutative diagram of bijections
    \begin{center}
    \begin{tikzpicture}
    \draw (-2,0)node{$\tp$} (2,1.5)node{$\ft$} (2,-1.5)node{$\sil$};
    \draw (2,2.5)node{$(\Fac_{m}(\XX),\XX^{\bot_{\leq 0}})$} (2,2)node[rotate=-90]{$\in$};
    \draw (2,-2.5)node{$\ppm(\XX)\oplus P[m]$} (2,-2)node[rotate=90]{$\in$};
    \draw (-3.7,0)node{$(\XX,P)\in$};
    \draw[->] (-.9,.2) tonode[above]{$\varphi$} (1.6,1);
    \draw[->] (-.9,-.2) tonode[below]{$\psi$} (1.6,-1);
    \draw[->] (2,-1) tonode[right]{$\chi$} (2,1);
    \draw[|->] (-3.7,.5) to (.5,2);
    \draw[|->] (-3.7,-.5) to (.5,-2);
    \draw[|->] (5.5,-1) to (5.5,.8);
    \draw (5.5,1.5)node{$\left(\Fac_{m}\left(H^{[-(m-1),0]}(\PP)\right),\right.$} (5.5,1)node{$\left.\Sub_{m}\left(H^{[-(m-1),0]}(\nu\PP[-1])\right)\right)$} (5.5,-1.5)node{$\PP$};
    \draw (3.1,1.5)node{$\ni$} (3.5,-1.5)node{$\ni$};
    \end{tikzpicture}
    \end{center}
\end{theorem}

To prove this theorem, we need some preparations.

\begin{lemma}\label{lem:pre=tau}
    Let $\XX,\YY\in\emod{m}{A}$ and $P\in\proj A$. For any $i\geq 1$, there are isomorphisms
    \begin{equation}\label{iso1}
        \Hom(\ppm(\XX),\ppm(\YY)[i])\cong\Hom(\ppm(\XX),\YY[i])\cong D\Hom(\YY,\AR{m}(\XX)[1-i]),
    \end{equation}
    and
    \begin{equation}\label{iso2}
        \Hom(P[m],\ppm(\YY)[i])\cong\Hom(P,\YY[i-m]).
    \end{equation}
\end{lemma}

\begin{proof}
    Denote $\PP=\ppm(\XX)$ and $\QQ=\ppm(\YY)$. Then $\XX\cong H^{[-(m-1),0]}(\PP)$ and $\YY\cong H^{[-(m-1),0]}(\QQ)$. By canonical truncation of $\QQ$, there is a triangle
    \[H^{-m}(\QQ)[m]\to\QQ\to\YY\to H^{-m}(\QQ)[m+1].\]
    Let $\RR$ be any one of $\PP$ and $P[m]$. Applying $\Hom(\RR,-)$ to this triangle, we get an exact sequence for any $i\geq 1$
    \[\begin{array}{rl}
         & \Hom(\RR,H^{-m}(\QQ)[m+i])\to\Hom(\RR,\QQ[i]) \to  \Hom(\RR,\YY[i])\\
         \to & \Hom(\RR,H^{-m}(\QQ)[m+i+1]),
    \end{array}\]
    where the first item and the last item are zero because $\RR$ is a complex of projective modules concentrated in degrees $\geq -m$ and $H^{-m}(\QQ)[m+i],H^{-m}(\QQ)[m+i+1]\in\D^{\leq -(m+1)}$. So we get isomorphisms
    $$\Hom(\RR,\QQ[i]) 
       \cong \Hom(\RR,\YY[i]),$$
    which implies the first isomorphism in \eqref{iso1} and the isomorphism in \eqref{iso2}.
    
    To show the second isomorphism in \eqref{iso1}, by the duality~\eqref{eq:Nakayama}, we have isomorphisms
    \begin{equation}\label{eq:lin541}
    \Hom(\PP,\YY[i]) \cong D\Hom(\YY[i],\nu\PP) \cong D\Hom(\YY[i-1],\nu\PP[-1]).
    \end{equation}
    Rotating the triangle~\eqref{eq:ARdef} and replacing $\ZZ$ with $\XX$, there is a triangle
    \[H^0(\nu\PP)[-2]\to\AR{m}(\XX)\to\nu\PP[-1]\to H^0(\nu\PP)[-1].\]
    Applying $\Hom(\YY[i-1],-)$ to this triangle, we get an exact sequence
    \[\begin{array}{rl}
         & \Hom(\YY[i-1],H^0(\nu\PP)[-2])\to\Hom(\YY[i-1],\AR{m}(\XX)) \\
        \to  & \Hom(\YY[i-1],\nu\PP[-1])\to\Hom(\YY[i-1],H^0(\nu\PP)[-1]),
    \end{array}
    \]
    where the first item and the last item are zero due to that $\YY[i-1]\in\D^{\leq 0}$ and $H^0(\nu\PP)[-2],H^0(\nu\PP)[-1]\in\D^{\geq 1}$. Therefore, we get isomorphisms
    \begin{equation}\label{eq:lin542}
        \Hom(\YY[i-1],\AR{m}(\XX))\cong \Hom(\YY[i-1],\nu\PP[-1]).
    \end{equation}
    Combining \eqref{eq:lin541} and \eqref{eq:lin542}, we get the second isomorphism in \eqref{iso1}.
\end{proof}

Recall that a complex $\PP$ in $K^b(\proj A)$ is called presilting if $\Hom(\PP,\PP[i])=0$ for any $i>0$.

\begin{proposition}\label{prop:pre=taup}
    There is a bijection from the set of (isoclasses of) basic positive $\AR{m}$-rigid pairs in $\emod{m}{A}$ to the set of (isoclasses of) basic $(m+1)$-term presilting complexes in $K^b(\proj A)$, sending $(\XX,P)$ to $\ppm(\XX)\oplus P[m]$.
\end{proposition}

\begin{proof}
    Let $\PP$ be an $(m+1)$-term complex in $K^b(\proj A)$. Write $\PP=\PP'\oplus P[m]$, where $P\in\proj A$ and $\PP'$ does not have $P'[m]$ as a direct summand for any nonzero $P'\in\proj A$. Then $\PP'$ is a minimal projective presentation of the object $H^{[-(m-1),0]}(\PP)$ in $\emod{m}{A}$. Hence, the map sending $(\XX,P)$ to $\ppm(\XX)\oplus P[m]$ is a bijection, with inverse sending $\PP=\PP'\oplus P[m]$ to $(H^{[-(m-1),0]}(\PP),P)$, from the set of (isoclasses of) basic pairs $(\XX,P)$ with $\XX\in\emod{m}{A}$ and $P\in\proj A$ to the set of (isoclasses of) basic $(m+1)$-term complexes in $K^b(\proj A)$. Thus, to complete the proof, we only need to show $(\XX,P)$ is positive $\AR{m}$-rigid if and only if $\ppm(\XX)\oplus P[m]$ is presilting. However, this follows directly from \Cref{lem:pre=tau}.
\end{proof}

Now we are ready to prove \Cref{thm:bi}.

\begin{proof}[Proof of \Cref{thm:bi}]
    To show $\psi$ is a bijection, by \Cref{prop:pre=taup}, we only need to show that for any positive $\AR{m}$-rigid pair $(\XX,P)$, it is $\AR{m}$-tilting if and only if the corresponding presilting complex $\ppm(\XX)\oplus P[m]$ is silting. Denote $\PP=\ppm(\XX)\oplus P[m]$. Then $H^{[-(m-1),0]}(\PP)=\XX$. Recall from \eqref{eq:tpp1} the definition of $\T(\PP)$. We claim 
    $${}^{\bot_{\leq 0}}(\AR{m}(\XX))\cap P^{\bot_{\leq 0}}=\T(\PP).$$
    Indeed, by the isomorphisms in \eqref{iso1}, we have ${}^{\bot_{\leq 0}}(\AR{m}(\XX))=\T(\ppm(X))$. Note that for any $\YY\in\emod{m}{A}$, $\Hom(P[m],\YY[i])=0$ for any $i\geq 1$ if and only if $\Hom(P[m],\YY[i])=0$ for any $1\leq i\leq m$ if and only if $\Hom(P,\YY[i])=0$ for any $i\leq 0$. So we have $P^{\bot_{\leq 0}}=\T(P[m])$. Thus, we have ${}^{\bot_{\leq 0}}(\AR{m}(\XX))\cap P^{\bot_{\leq 0}}=\T(\ppm(X))\cap\T(P[m])=\T(\PP)$.
    So we only need to show that $\PP$ is silting if and only if $\T(\PP)\subseteq\Fac_{m}(H^{[-(m-1),0]}(\PP))$. The ``only if" part follows directly from \Cref{prop:genTP}. For the ``if" part, for any object $\MM\in\T(\PP)\subseteq\Fac_{m}(H^{[-(m-1),0]}(\PP))$, by definition of $m$-factors, there is a triangle
    $$\ZZ\to\TT_0\to\MM\to\ZZ[1],$$
    with $\TT_0\in\add H^{[-(m-1),0]}(\PP)$ and $\ZZ\in \emod{m}{A}$. So by \Cref{new}, $\PP$ is silting. 
    Thus, $\psi$ is a bijection.

    By the bijection in \Cref{rmk:G} and using \Cref{prop:genTP}, $\chi$ is a bijection. Let $\varphi=\chi\circ\psi$ and $(\T,\F)=\varphi(\XX,P)$. Then $$\T=\chi(\ppm(\XX)\oplus P[m])=\Fac_m(H^{[-(m-1),0]}(\ppm(\XX)\oplus P[m]))=\Fac_m(\XX).$$
    By the isomorphisms~\eqref{eq:iso} in \Cref{lem:bz1-2.7}, we have $$\F=\F(\ppm(\XX)\oplus P[m])=\{\YY\in\emod{m}{A}\mid\Hom(\XX,\YY[j])=0,\ j\leq 0\}=\XX^{\bot_{\leq0}}.$$ Thus, the bijection $\varphi$ has the form shown in the theorem.
\end{proof}

\begin{corollary}
    Let $\XX$ be a $\AR{m}$-tilting object in $\emod{m}{A}$. Then $$(\Fac_{m}(\XX),\Sub_{m}(\AR{m}(\XX)))$$ is a functorially finite $s$-torsion pair in $\emod{m}{A}$.
\end{corollary}

\begin{proof}
    By definition, $(\XX,0)$ is a $\AR{m}$-tilting pair in $\emod{m}{A}$. Then by \Cref{thm:bi}, there is a corresponding functorially finite $s$-torsion pair in $\emod{m}{A}$
    $$(\Fac_{m}(\XX),\Sub_{m}(H^{[-(m-1),0]}(\nu\ppm(\XX)[-1])).$$
    However, by \Cref{rmk:AR}, $H^{[-(m-1),0]}(\nu\ppm(\XX)[-1])=\AR{m}(\XX)$.
\end{proof}

For any $\XX\in\emod{m}{A}$, we denote by $|\XX|$ the number of non-isomorphic indecomposable direct summands of $\XX$.

\begin{corollary}\label{num}
    Let $(\XX,P)$ be a $\AR{m}$-tilting pair in $\emod{m}{A}$. Then $|\XX|+|P|=|A|$.
\end{corollary}

\begin{proof}
    Since $\ppm(\XX)$ is a minimal projective presentation of $\XX\in\emod{m}{A}$, we have $|\ppm(\XX)|=|\XX|$. So $|\XX|+|P|=|\ppm(\XX)|+|P[m]|=|\ppm(\XX)\oplus P[m]|$. By \Cref{thm:bi}, $\ppm(\XX)\oplus P[m]$ is a silting complex. Then by \cite[Corollary~2.28]{AI}, $|\ppm(\XX)\oplus P[m]|=|A|$. Thus, we get the required equality.
\end{proof}

The following result tells us that for a $\AR{m}$-tilting pair, its right part is determined by its left part.

\begin{corollary}\label{unique}
    Let $(\XX,P)$ be a $\AR{m}$-tilting pair in $\emod{m}{A}$. Then for any $Q\in\proj A$, if $(\XX,Q)$ is a positive $\AR{m}$-rigid pair, then $Q\in\add P$.
\end{corollary}

\begin{proof}
    By definition, $(\XX,P\oplus Q)$ is a positive $\AR{m}$-rigid pair. Since $(\XX,P)$ is a $\AR{m}$-tilting pair, we have ${}^{\bot_{\leq 0}}(\AR{m}(\XX))\cap (P\oplus Q)^{\bot_{\leq 0}}\subseteq{}^{\bot_{\leq 0}}(\AR{m}(\XX))\cap P^{\bot_{\leq 0}}\subseteq\Fac_{m}(\XX)$. So $(\XX,P\oplus Q)$ is also a $\AR{m}$-tilting pair. Thus, by \Cref{num}, we have $|P\oplus Q|=|A|-|\XX|=|P|$, which implies $Q\in\add P$.
\end{proof}

By \Cref{rmk:G} and \Cref{prop:tpp1}, for any functorially finite $s$-torsion pair $(\T,\F)$ in $\emod{m}{A}$, both $\T$ and $\F$ have enough projective objects and enough injective objects.

\begin{corollary}\label{cor:num}
    Let $(\T,\F)$ be a functorially finite $s$-torsion pair in $\emod{m}{A}$. Then the number of projective objects in $\T$ (resp. $\F$) is the same as the number of injective objects in $\T$ (resp. $\F$).
\end{corollary}

\begin{proof}
    We only show the assertion for $\T$ since the assertion for $\F$ can be shown similarly. By \Cref{thm:bi}, there is an $(m+1)$-term silting complex $\PP$ in $K^b(\proj A)$ and a $\AR{m}$-tilting pair $(\XX,P)$ in $\emod{m}{A}$ such that $\T=\T(\PP)=\Fac_m\XX$ and $\XX=H^{[-(m-1),0]}(\PP)$. By \Cref{prop:tpp1}, the number of non-isomorphic indecomposable projective objects in $\T(\PP)$ is $|\XX|$, which by \Cref{num} equals $|A|-|P|$. Also by \Cref{prop:tpp1}, the number of non-isomorphic indecomposable injective objects in $\T(\PP)$ is $|\mathfrak{t}(\nu A[m-1])|$.
    
    We claim that for any indecomposable direct summand $Q$ of $A$, the pair $(\XX,Q)$ is positive $\AR{m}$-rigid if and only if $\nu Q[m-1]\in\F(\PP)$. By the duality~\eqref{eq:Nakayama}, we have isomorphisms
    $$\Hom(Q[m-1],\XX[i])\cong D\Hom(\XX,(\nu Q[m-1])[-i]),\ i\geq 0.$$
    Since $\Hom(Q[m-1],\XX[i])=0$ for any $i<0$, $(\XX,Q)$ is positive $\AR{m}$-rigid if and only if $\Hom(Q[m-1],\XX[i])=0$ for any $i\geq 0$ if and only if $\Hom(\XX,(\nu Q[m-1])[i])=0$ for any $i\leq 0$ if and only if $\nu Q[m-1]\in\XX^{\bot_{\leq 0}}=\F(\PP)$. Thus, the claim is proved.
    
    So, by \Cref{unique}, $Q\in\add P$ if and only if $\mathfrak{t}(\nu Q[m-1])=0$. Thus, to complete the proof, it suffices to show that if $Q\notin\add P$, then $\mathfrak{t}(\nu Q[m-1])$ is indecomposable. Consider the canonical triangle of $\nu Q[m-1]$ with respect to $(\T,\F)$
    \begin{equation}\label{new1}
        \mathfrak{t}(\nu Q[m-1])\xrightarrow{f} \nu Q[m-1]\to \mathfrak{f}(\nu Q[m-1])\to \mathfrak{t}(\nu Q[m-1])[1].
    \end{equation}
    Since $\nu Q[m-1]$ is injective in $\emod{m}{A}$, the morphism $f$ is a left $\add \nu A[m-1]$-approximation of $\mathfrak{t}(\nu Q[m-1])$. By \Cref{tf}, $f$ is a right minimal $\T$-approximation of $\nu Q[m-1]$. In particular, $f\neq 0$. This, together with $\mathfrak{t}(\nu Q[m-1])\neq 0$ and $\nu Q[m-1]$ is indecomposable, shows that $f$ is left minimal. Suppose, to the contrary, that $\mathfrak{t}(\nu Q[m-1])=\XX_1\oplus\XX_2$ with nonzero objects $\XX_1$ and $\XX_2$. Write $f=\begin{pmatrix}f_1&f_2\end{pmatrix}$ with $f_i:\XX_i\to \nu Q[m-1]$, $i=1,2$. Since $f$ is right minimal, both $f_1$ and $f_2$ are not zero. So $f_1$ and $f_2$ are left minimal, since $\nu Q[m-1]$ is indecomposable. Let $g=\begin{pmatrix}f_1&0\end{pmatrix}:\mathfrak{t}(\nu Q[m-1])=\XX_1\oplus\XX_2\to \nu Q[m-1]$. Since $f$ is a left $\add\nu A[m-1]$-approximation of $\mathfrak{t}(\nu Q[m-1])$, there is a morphism $h:\nu Q[m-1]\to \nu Q[m-1]$ such that $h\circ f=g$. So we have $h\circ f_1=f_1$ and $h\circ f_2=0$. Since $f_1$ is left minimal, $h$ is an isomorphism. Thus, we have $f_2=0$, a contradiction. Hence, $\mathfrak{t}(\nu Q[m-1])$ is indecomposable.
\end{proof}

In the usual case (i.e., $m=1$), for any positive $\AR{m}$-rigid pair $(\XX,P)$ in $\emod{m}{A}$, the following hold.
\begin{enumerate}
    \item The pair $(\Fac_m(\XX),\Fac_m(\XX)^\bot)$ is a functorially finite $s$-torsion pair in $\emod{m}{A}$ \cite[Theorem~5.10]{AS}.
    \item There exist $\YY\in\emod{m}{A}$ and $Q\in\proj A$ such that $(\XX\oplus\YY,P\oplus Q)$ is $\AR{m}$-tilting \cite[Proposition~2.3~(a) and Theorem~2.10]{AIR}.
    \item If $(\XX,P)$ is maximal positive $\AR{m}$-rigid in the sense that any indecomposable positive $\AR{m}$-rigid pair $(\YY,Q)$ such that $(\XX\oplus\YY,P\oplus Q)$ is positive $\AR{m}$-rigid is isomorphic to a direct summand of $(\XX,P)$, then $(\XX,P)$ is $\tau$-tilting \cite[Corollary~2.13]{AIR}.
\end{enumerate}
However, these three are not true in general.

\begin{example}
    Let $A=\k Q/I$, where $Q$ is the quiver 
    $$\xymatrix{1 \ar@/^0.5pc/[r]^{x_1} \ar@/_0.5pc/[r]_{y_1} & 2 \ar@/^0.5pc/[r]^{x_2} \ar@/_0.5pc/[r]_{y_2} & 3}$$ and $I=\langle x_1x_2, y_1y_2\rangle$. Let $M$ be the representation
    $$\xymatrix{\k \ar@/^0.5pc/[r]^{1} \ar@/_0.5pc/[r]_{0} & \k \ar@/^0.5pc/[r]^{0} \ar@/_0.5pc/[r]_{1} & \k }.$$ 
    Consider the object $\XX=(0\to M)$ in $\emod{2}{A}$. Its minimal projective presentation 
    $$\pp_2(\XX)=(P_3\xrightarrow{y_2}P_2\xrightarrow{y_1}P_1)$$
    is shown in \cite{LZ} to be indecomposable but not to be a proper direct summand of a presilting complex in $K^b(\proj A)$. Thus, by the bijection $\psi$ in \Cref{thm:bi}, $(\XX,0)$ is maximal positive $\AR{2}$-rigid and there is no positive $\AR{2}$-rigid pair $(\YY,Q)$ such that $(\XX,P)\oplus(\YY,Q)$ is $\AR{2}$-tilting.
    
    We shall show that the pair $(\Fac_2(\XX),\Fac_2(\XX)^\bot)$ is not a functorially finite $s$-torsion pair. Suppose, to the contrary, that it is. Then by \Cref{thm:bi}, there is a $\AR{2}$-tilting pair $(\XX',P)$ such that $\Fac_2(\XX)=\Fac_2(\XX')$. For any $\YY\in\Fac_2(\XX)$, by inclusion~\eqref{eq:inc}, $\YY\in{}^{\bot_{\leq 0}}(\AR{2}(\XX))$. So $\Hom(\YY,\AR{2}(\XX))=0$, which by \Cref{cor:AR}, implies $\mathbb{E}(\XX,\YY)=0$. Hence, $\XX$ is projective in $\Fac_2(\XX)$. So by \Cref{prop:tpp1}, $\XX$ is a direct summand of $\XX'$. Thus, $\pp_2(\XX)$ is a direct summand of $\pp_2(\XX')$. However, by \Cref{thm:bi}, $\pp_2(\XX')\oplus\nu P[2]$ is a silting complex. This is a contradiction.
\end{example}

\appendix

\section{Truncation}\label{sec: trun}

Let $A$ be a finite-dimensional algebra over a field $\k$, and $\mod A$ the category of finitely generated right $A$-modules. In this appendix, we collect some basic constructions about truncating complexes of modules in $\mod A$, mainly for the constructions of the Auslander-Reiten translations $\AR{m}$ and $\AR{m}^-$ in \Cref{sec:AR}.

For any complex $\XX=(X^i,d^i:X^i\to X^{i+1})$ of modules in $\mod A$, there are the following two kinds of truncation. Let $p$ be an integer.

The \emph{stupid truncation} $\sigma_{\leq p}(\XX)$ and $\sigma_{\geq p}(\XX)$ are defined as:
\[\begin{array}{cccccccccccc}
    \sigma_{\leq p}(\XX) & = & (\cdots & \xrightarrow{d^{p-2}} & X^{p-1} & \xrightarrow{d^{p-1}} & X^{p} & \to & 0 & \to & \cdots),\\
    \sigma_{\geq p}(\XX) & = & (\cdots & \to & 0 & \to &  X^{p} & \xrightarrow{d^p} & X^{p+1} & \xrightarrow{d^{p+1}} & \cdots).
\end{array}\]
There is a triangle in $D(\mod A)$:
\[\sigma_{\geq p+1}(\XX)\to\XX\to \sigma_{\leq p}(\XX)\to \sigma_{\geq p+1}(\XX)[1].\]
The \emph{canonical truncation} $\sigma^{\leq p}(\XX)$ and $\sigma^{\geq p}(\XX)$ are defined as
\[\begin{array}{ccccccccccc}
    \sigma^{\leq p}(\XX)=&(\cdots&\xrightarrow{d^{p-2}} & X^{p-1} & \xrightarrow{d^{p-1}} & \ker d^p & \to &  0 & \to & \cdots),\\
    \sigma^{\geq p}(\XX)  = & (\cdots & \to & 0 & \to & \coker d^{p-1} & \xrightarrow{d^p} & X^{p+1} & \xrightarrow{d^{p+1}} &\cdots).
\end{array}\]
There is a triangle in $D(\mod A)$:
\[\sigma^{\leq p}(\XX)\to\XX\to \sigma^{\geq p+1}(\XX)\to \sigma^{\leq p}(\XX)[1].\]

\end{document}